\documentclass[11pt,a4paper]{amsart}
\usepackage[latin1]{inputenc}
\usepackage{amsmath}
\usepackage{amsfonts}
\usepackage{amssymb}
\usepackage{amsthm}
\usepackage[usenames, dvipsnames]{color}
\usepackage{tikz}
\usepackage[all]{xy}
\usepackage{geometry}
\usepackage[normalem]{ulem}
\usepackage{soul}

\definecolor{liens}{rgb}{1,0,0}
\usepackage[colorlinks=true, linkcolor=blue, 
hyperfootnotes=true,citecolor=blue,urlcolor=black]{hyperref}
\definecolor{cadmiumgreen}{rgb}{0.0, 0.42, 0.24}
\definecolor{msviolet}{rgb}{1.0, 0.0, 1.0}

\newtheorem*{proposition*}{Proposition~\ref{prop:singcases}}
\newtheorem*{lem**}{Lemma~\ref{theo:doublezero}}
\newtheorem*{thm**}{Theorem~\ref{theo:maintheotransc}}

\newtheorem{thm}{Theorem}[section]
\newtheorem{cor}[thm]{Corollary}
\newtheorem{lemma}[thm]{Lemma}
\newtheorem{lem}[thm]{Lemma}
\newtheorem{prop}[thm]{Proposition}

\theoremstyle{definition}

\newtheorem{defi}[thm]{Definition} 

\newenvironment{prf}[1]{\trivlist,
\item[\hskip \labelsep{\bf #1.\hspace*{.3em}}]}{~\hspace{\fill}~$\square$\endtrivlist}

\theoremstyle{remark}
\newtheorem{rmk}[thm]{Remark}
\newtheorem{rem}[thm]{Remark} 

\numberwithin{equation}{section}
\def\iup{{\tilde{\iota}}}

\def\Z{\mathbb{Z}}
\def\C{\mathbb{C}}
\def\R{\mathbb{R}}
\def\Q{\mathbb{Q}}

\def\P1{\mathbb{P}^{1}}

\def\beq{\begin{equation}}
\def\eeq{\end{equation}}

\def\Etproj{\overline{E_t}}

\def\CX{{\mathbb C}}
\def\RX{{\mathbb R}}

\def\P1{\mathbb{P}^{1}}

\def\a{{\alpha}}
\def\b{{\beta}}

\begin{document}
\title{{On the kernel curves associated with walks in the quarter plane}}
\author{Thomas Dreyfus}
\address{Institut de Recherche Math\'ematique Avanc\'ee, U.M.R. 7501 Universit\'e de Strasbourg et C.N.R.S. 7, rue Ren\'e Descartes 67084 Strasbourg, FRANCE}
\email{dreyfus@math.unistra.fr}
\author{Charlotte Hardouin}
\address{Universit\'e Paul Sabatier - Institut de Math\'ematiques de Toulouse, 118 route de Narbonne, 31062 Toulouse.}
\email{hardouin@math.univ-toulouse.fr}
\author{Julien Roques}
\address{Univ Lyon, Universit\'e Claude Bernard Lyon 1, CNRS UMR 5208, Institut Camille Jordan, 43 blvd. du 11 novembre 1918, F-69622 Villeurbanne cedex, France}
\email{roques@math.univ-lyon1.fr}
\author{Michael F. Singer }
\address{Department of Mathematics, North Carolina State University,
Box 8205, Raleigh, NC 27695-8205, USA}
\email{singer@ncsu.edu}


\thanks{This project has received funding from the European Research Council (ERC) under the European Union's Horizon 2020 research and innovation programme under the Grant Agreement No 648132. The authors would like to thank ANR-19-CE40-0018, ANR-13-JS01-0002-01, ANR-11-LABX-0040-CIMI within
the program ANR-11-IDEX-0002-0, LabEx PERSYVAL-Lab ANR-11-LABX-0025-01, the Simons Foundation (\#349357, Michael Singer).}

\date{\today}

\bibliographystyle{amsalpha} 

\begin{abstract} 
The {\it kernel method} is an essential tool for the study of generating series of {\it walks} in the quarter plane. This method involves equating to zero  a certain polynomial - the {\it kernel polynomial} - and using properties of the curve - the {\it kernel curve} - this defines. In the present paper, we investigate the basic properties of the {\it kernel curve} (irreducibility, singularities, genus, uniformization, etc).
\end{abstract}

\maketitle
\setcounter{tocdepth}{1}
\tableofcontents

\sloppy

\section*{Introduction} 
 
Consider a {\it walk} with small steps in the positive quadrant $\mathbb{Z}_{\geq0}^{2}=\{0,1,2,\ldots\}^2$ starting from $P_0:=(0,0)$, that is a succession of points
\begin{equation*}
     P_{0}, P_1,\ldots ,P_k,
\end{equation*}   
where each $P_n$ lies in the quarter plane, where the moves (or steps) $P_{n+1}-P_{n}$ belong to $\{0,\pm 1\}^{2}$, and the probability to move in the direction $P_{n+1}-P_{n}=(i,j)$ may be interpreted as  some {\it weight-parameter} $d_{i,j}\in [0,1]$, with $\sum_{(i,j)\in\{0,\pm 1\}^{2}}d_{i,j}=1$.  The step set or the \emph{ model} of the {\it walk } is  the set of directions with nonzero {\it weights}, that is 
$$
\mathcal{S}=\{ (i,j ) \in \{0,\pm 1\}^{2}  \ | \ d_{i,j} \neq 0\}.
$$ The following picture is an example of  such path:\smallskip
\begin{center}
\begin{tikzpicture}[scale=.8, baseline=(current bounding box.center)]
\foreach \x in {0,1,2,3,4,5,6,7,8,9,10}
  \foreach \y in {0,1,2,3,4}
    \fill(\x,\y) circle[radius=0pt];
\draw (0,0)--(10,0);
\draw (0,0)--(0,4);
\draw[->](0,0)--(1,1);
\draw[->](1,1)--(1,0);
\draw[->](1,0)--(0,1);
\draw[->](0,1)--(1,2);
\draw[->](1,2)--(2,1);
\draw[->](2,1)--(2,0);
\draw[->](2,0)--(3,1);
\draw[->](3,1)--(3,0);
\draw[->](3,0)--(4,1);
\draw[->](4,1)--(3,2);
\draw[->](3,2)--(2,3);
\draw[->](2,3)--(2,2);
\draw[->](2,2)--(3,3);
\draw[->](3,3)--(4,2);
\draw[->](4,2)--(4,1);
\draw[->](4,1)--(5,0);
\draw[->](5,0)--(6,1);
\draw[->](6,1)--(6,0);
\draw[->](6,0)--(7,1);
\draw[->](7,1)--(8,0);
\draw[->](8,0)--(9,1);
\draw[->](9,1)--(9,0);
\draw[->](9,0)--(10,1);
\draw[->](10,1)--(9,2);
\draw[->](9,2)--(8,3);
\draw[->](8,3)--(8,2);
\end{tikzpicture}
\quad\quad
$\mathcal{S}=\left\{\begin{tikzpicture}[scale=.4, baseline=(current bounding box.center)]
\foreach \x in {-1,0,1} \foreach \y in {-1,0,1} \fill(\x,\y) circle[radius=2pt];
\draw[thick,->](0,0)--(-1,1);
\draw[thick,->](0,0)--(1,1);
\draw[thick,->](0,0)--(1,-1);
\draw[thick,->](0,0)--(0,-1);
\end{tikzpicture}\right\}$
\end{center}
\medskip
 Such objects are very natural both in combinatorics and probability theory: they are interesting for themselves and also because they are strongly related to other discrete structures, see \cite{BMM,DeWa-15} and references therein.\par 
If $d_{0,0}=0$ and if the nonzero $d_{i,j}$ all have the same value, we say that  the model  is {\it unweighted}.\par
 The {\it weight of a given walk} is defined to be the product of the {\it weights} of its component steps. For any $(i,j)\in \Z_{\geq 0}^{2}$ and any $k\in \Z_{\geq 0}$, we let $q_{i,j,k}$ be the sum of the {\it weights} of all {\it walks} reaching  the position $(i,j)$ from the initial position $(0,0)$ after $k$ steps. We introduce the corresponding trivariate generating series\footnote{In several papers it is not assumed that $\sum_{i,j} d_{i,j}=1$. But after a rescaling of the $t$ variable, we may always reduce to the case  $\sum_{i,j} d_{i,j}=1$.}
$$
Q(x,y,t):=\displaystyle \sum_{i,j,k\geq 0}q_{i,j,k}x^{i}y^{j}t^{k}.
$$
The study of the nature of this generating series has attracted the attention of many authors, see for instance \cite{BostanKauersTheCompleteGenerating,BRS,BBMR15, BBMR17,BMM,DHRS,dreyfus2017walks,dreyfus2017differential,dreyfus2019length,
KurkRasch, Mishna09,MR09,MelcMish,RaschelJEMS}. The typical questions are: is $Q(x,y,t)$ rational, algebraic, holonomic,  etc? The starting point of most of these works is the following functional equation, see for instance \cite[Lemma 1.1]{dreyfus2017walks}, and \cite{BMM}  for the {\it unweighted} case
$$
K(x,y,t)Q(x,y,t)=xy +K(x,0,t)Q(x,0,t) +K(0,y,t)Q(0,y,t)+td_{-1,-1} Q(0,0,t)
$$
where 
$$
K(x,y,t)=xy (1-t S(x,y))
$$
with 
$$
S(x,y) =\sum_{(i,j)\in \{0,\pm 1\}^{2}} d_{i,j}x^i y^j. 
$$ 
The polynomial $K(x,y,t)$
 is called the {\it kernel polynomial} and is the main character of the {\it kernel method}.\par 

 Roughly speaking, the first step of the {\it kernel method} consists in ``eliminating'' the left hand side of the above functional equation by restricting our attention to the $(x,y)$ such that $K(x,y,t)=0$. The set $E_{t}$ made of the $(x,y)$ such that $K(x,y,t)=0$ is called the {\it kernel curve}:  
$$
 E_t = \{(x,y) \in \C \times \C \ \vert \ K(x,y,t) = 0\}. 
 $$ 
Thus, for $(x,y) \in E_{t}$, one has 
\begin{equation}\label{eq1}\tag{1}
0=xy +K(x,0,t)Q(x,0,t) +K(0,y,t)Q(0,y,t)+td_{-1,-1} Q(0,0,t),
\end{equation}
provided that the various series can be evaluated at the given points.

The second step of the {\it kernel method} is to exploit certain involutive birational transformations $\iota_{1},\iota_{2}$ (they are called $\zeta,\eta$ in \cite{FIM}) of the {\it kernel curve} $E_{t}$ of the form 
$$
\iota_{1}(x,y) = (x,y') \text{ and } \iota_{2}(x,y) = (x',y)
$$ 
in order to deduce from \eqref{eq1} some functional equations for $Q(x,0,t)$ and $Q(0,y,t)$. Hence $\iota_{1}$ and $\iota_{2}$ switch the roots of the degree two polynomials $y\mapsto K(x,y,t)$ and $x\mapsto K(x,y,t)$ respectively.
Concretely, the birational transformations $\iota_{1},\iota_{2}$ are induced  by restriction to the curve of the involutive  birational transformations $i_{1},i_{2}$ of $\CX^{2}$ given by
$$
i_1(x,y) =\left(x, \frac{A_{-1}(x) }{A_{1}(x)y}\right) \text{ and }  i_2(x,y)=\left(\frac{B_{-1}(y)}{B_{1}(y)x},y\right)
$$
where the $A_{i}(x) \in x^{-1}\Q[x]$ and the $B_{i}(y) \in y^{-1}\Q[y]$ are defined by  
$$
S(x,y) = A_{-1}(x) \frac{1}{y} +A_{0}(x)+ A_{1}(x) y
= B_{-1}(y) \frac{1}{x} +B_{0}(y)+ B_{1}(y) x,
$$
see \cite[Section 3]{BMM}, \cite[Section 3]{KauersYatchak}  or \cite{FIM}. 
 These $i_{1}$ and $i_{2}$ are the generators of the group of the {\it walk}; see \cite{BMM} for details. Note that although $i_1$ and $i_2$ do not depend on $t$, the group generated by the induced involutions $\iota_1$ and $\iota_2$ may depend on $t$, since the order of $\iota_2 \circ \iota_1$ may depend upon $t$, see Remark \ref{rem2}.\par 
The third step of the {\it kernel method} is to use the above mentioned functional equations of $Q(x,0,t)$ and $Q(0,y,t)$ to continue these series as multivalued meromorphic functions. To perform this step, we need an explicit uniformization of the {\it kernel curve}.\par 
The aim of the present paper is to study the {\it kernel curve} $E_{t}$ and the birational transformations $\iota_{1},\iota_{2}$.  
Note that a similar study has been done in the case $t=1$ in \cite{FIM} and in the {\it unweighted} case in  \cite{KurkRasch}. The goal of the present paper is to extend these works to the {\it weighted} case when $t\in ]0,1[$ is transcendental over $\Q(d_{i,j})$. Although many results are similar to \cite{FIM}, the proofs are different. The assumptions we  make on $t$ are crucial in many parts of the proof and it is not clear how the proofs of \cite{FIM} exactly pass to this context.  \par 

 We could expect to have classification of the geometric properties of $\Etproj$ involving configurations of {\it weights} independent of $t$.  Hopefully, this paper has been followed by 
 \cite{dreyfus2017differential} where 
 the case for general $t\in ]0,1[$ has been considered. The proofs of the latter paper use continuity arguments with respect the parameter $t$ that permit to deduce many results for algebraic values of $t$.  Such reasoning needs to be very cautious, and it is not trivial to deduce the results for general $t\in ]0,1[$ from the $t=1$ case.  We will mention explicitly every time if the results are correct for arbitrary values of $t\in ]0,1[$.\\
 
The paper is organized as follows. In Section \ref{sec1.3}, we describe the {\it nondegenerate  models of walks}. In Section \ref{sec:algcurvedefkernel}, we determine the singularities and the genus of the {\it kernel curve}.  In Section \ref{appendix:involutions}, we establish the basic properties of $\iota_{1}$ and $\iota_{2}$. Finally, in Section \ref{appendix:param}, we give an explicit uniformization of the {\it kernel curve}.\\

\textbf{Acknowledgment.}  The authors want to warmly thank the referees for their  detailed and helpful comments.
\section{Nondegenerate  walks}\label{sec1.3}

From now on, we fix $t\in ]0,1[$, that is  transcendental over the field $\Q(d_{i,j})$. We start by recalling the notion of {\it {degenerate walks}  introduced in \cite{FIM}}.  
 
 \begin{defi}
A  {\it model of walk} is called {\it degenerate} if one of the following holds:
\begin{itemize}
\item $K(x,y,t)$ is reducible as an element of the polynomial ring $\C[x,y]$, 
\item $K(x,y,t)$ has $x$-degree less than or equal to $1$,
\item $K(x,y,t)$ has $y$-degree less than or equal to $1$.
\end{itemize}
 \end{defi}
 
In what follows we will sometimes represent a {\it model of walks} with arrows. We will also use dashed arrows for a family of models. For instance, the family of models  represented by 
$$\begin{tikzpicture}[scale=0.6, baseline=(current bounding box.center)]
\draw[dashed,->,line width=0.5mm](0,0)--(-1,-1);
\draw[thick,->](0,0)--(1,-1);
\draw[thick,->](0,0)--(1,1);
\draw[thick,->](0,0)--(0,1);
\end{tikzpicture} \hbox{ or } \left\{\begin{tikzpicture}[scale=.6, baseline=(current bounding box.center)]
\draw[thick,->](0,0)--(1,1);
\end{tikzpicture}, \begin{tikzpicture}[scale=.6, baseline=(current bounding box.center)]
\draw[thick,->](0,0)--(1,-1);
\end{tikzpicture}, \begin{tikzpicture}[scale=.6, baseline=(current bounding box.center)]
\draw[thick,->](0,0)--(0,1);
\end{tikzpicture}, \begin{tikzpicture}[scale=.6, baseline=(current bounding box.center)]
\draw[dashed,->,line width=0.5mm](0,0)--(-1,-1);
\end{tikzpicture} \right\},$$
correspond to models with $d_{1,1},d_{1,-1},d_{0,1}\neq 0$, $d_{1,0}=d_{0,-1}=d_{-1,1}=d_{-1,0}=0$, and where nothing is assumed on $d_{-1,-1}$ and $d_{0,0}$. In the following results, 
 the behavior of the {\it kernel curve} never depends on $d_{0,0}$. This is the reason why, to reduce the amount of notations, we have decided to not associate an arrow to $d_{0,0}$.
The following result is the analog of \cite[Lemma 2.3.2]{FIM}, that focuses on  the case $t=1$.  Our proof differs from   the proof of \cite[Lemma 2.3.2]{FIM},  which only considered factorization over $\R[x,y]$, while in this paper,  we need to prove the absence of factorization over $\C[x,y]$.

  \begin{prop}\label{prop:singcases}
A  {\it model of walk} is {\it degenerate} if and only if at least one of the following holds: 
\begin{enumerate}
\item \label{case1}There exists $i\in \{- 1,1\}$ such that $d_{i,-1}=d_{i,0}=d_{i,1}=0$. This corresponds to the following  families of  {\it models of walks}

$$\begin{tikzpicture}[scale=0.6, baseline=(current bounding box.center)]
\foreach \x in {-1,0,1} \foreach \y in {-1,0,1} \fill(\x,\y) circle[radius=0pt];
\draw[dashed,->,line width=0.5mm](0,0)--(0,-1);
\draw[dashed,->,line width=0.5mm](0,0)--(1,-1);
\draw[dashed,->,line width=0.5mm](0,0)--(1,0);
\draw[dashed,->,line width=0.5mm](0,0)--(1,1);
\draw[dashed,->,line width=0.5mm](0,0)--(0,1);
\end{tikzpicture} \ \ ,\quad 
\begin{tikzpicture}[scale=0.6, baseline=(current bounding box.center)]
\foreach \x in {-1,0,1} \foreach \y in {-1,0,1} \fill(\x,\y) circle[radius=0pt];
\draw[dashed,->,line width=0.5mm](0,0)--(0,-1);
\draw[dashed,->,line width=0.5mm](0,0)--(-1,-1);
\draw[dashed,->,line width=0.5mm](0,0)--(-1,0);
\draw[dashed,->,line width=0.5mm](0,0)--(-1,1);
\draw[dashed,->,line width=0.5mm](0,0)--(0,1);
\end{tikzpicture} $$

\item \label{case2} There exists $j\in \{-1, 1\}$ such that $d_{-1,j}=d_{0,j}=d_{1,j}=0$. This corresponds to the following   families of  {\it models of walks}

$$\begin{tikzpicture}[scale=.6, baseline=(current bounding box.center)]
\foreach \x in {-1,0,1} \foreach \y in {-1,0,1} \fill(\x,\y) circle[radius=0pt];
\draw[dashed,->,line width=0.5mm](0,0)--(-1,0);
\draw[dashed,->,line width=0.5mm](0,0)--(-1,-1);
\draw[dashed,->,line width=0.5mm](0,0)--(0,-1);
\draw[dashed,->,line width=0.5mm](0,0)--(1,-1);
\draw[dashed,->,line width=0.5mm](0,0)--(1,0);
\end{tikzpicture} \ \ ,
 \quad
\begin{tikzpicture}[scale=.6, baseline=(current bounding box.center)]
\foreach \x in {-1,0,1} \foreach \y in {-1,0,1} \fill(\x,\y) circle[radius=0pt];
\draw[dashed,->,line width=0.5mm](0,0)--(-1,0);
\draw[dashed,->,line width=0.5mm](0,0)--(-1,1);
\draw[dashed,->,line width=0.5mm](0,0)--(0,1);
\draw[dashed,->,line width=0.5mm](0,0)--(1,1);
\draw[dashed,->,line width=0.5mm](0,0)--(1,0);
\end{tikzpicture}  $$

\item \label{case3} All the {\it weights} are $0$ except maybe  $\{d_{1,1},d_{0,0},d_{-1,-1}\}$ or  $\{d_{-1,1},d_{0,0},d_{1,-1}\}$. This corresponds to the following  families of  {\it models of walks}
$$
\left\{\begin{tikzpicture}[scale=.6, baseline=(current bounding box.center)]
\draw[dashed,->,line width=0.5mm](0,0)--(1,1);
\end{tikzpicture}, \begin{tikzpicture}[scale=.6, baseline=(current bounding box.center)]
\draw[dashed,->,line width=0.5mm](0,0)--(-1,-1);
\end{tikzpicture} \right\} \ ,
\quad
\left\{\begin{tikzpicture}[scale=.6, baseline=(current bounding box.center)]
\draw[dashed,->,line width=0.5mm](0,0)--(-1,1);
\end{tikzpicture}, \begin{tikzpicture}[scale=.6, baseline=(current bounding box.center)]
\draw[dashed,->,line width=0.5mm](0,0)--(1,-1);
\end{tikzpicture} \right\}
$$
\end{enumerate}
\end{prop} 

\begin{proof} This proof is organized as follows.  We begin by showing that (\ref{case1}) (resp. (\ref{case2})) corresponds to $K(x,y,t)$ having $x$-degree $\leq 1$ or  $x$-valuation $\geq 1$ (resp.  $y$-degree $\leq 1$ or $y$-valuation $\geq 1$). In these cases, {\it the model of the walk} is clearly {\it degenerate}. Assuming (\ref{case1}) and (\ref{case2})  do not hold, we then show that (\ref{case3}) holds if and only if $K(x,y,t)$ is reducible.

\underline{Cases (\ref{case1}) and (\ref{case2}).} It is clear that $K(x,y,t)$ has $x$-degree $\leq 1$ if and only if  ${d_{1,-1}=d_{1,0}=d_{1,1}=0}$. Similarly,  $K(x,y,t)$ has $y$-degree $\leq 1$ if and only if we have ${d_{-1,1}=d_{0,1}=d_{1,1}=0}$. Furthermore, ${d_{-1,-1}=d_{-1,0}=d_{-1,1}=0}$ if and only if $K(x,y,t)$ has $x$-valuation $\geq 1$. Similarly, $d_{-1,-1}=d_{0,-1}=d_{1,-1}=0$ if and only if $K(x,y,t)$ has $y$-valuation $\geq 1$. In these cases, {\it the  model of the walk} is clearly {\it degenerate}.

\underline{Case (\ref{case3}).} We now assume that cases (\ref{case1}) and (\ref{case2}) do not hold. This implies that the model belongs to the family of models $\left\{\begin{tikzpicture}[scale=.6, baseline=(current bounding box.center)]
\draw[dashed,->,line width=0.5mm](0,0)--(1,1);
\end{tikzpicture}, \begin{tikzpicture}[scale=.6, baseline=(current bounding box.center)]
\draw[dashed,->,line width=0.5mm](0,0)--(-1,-1);
\end{tikzpicture} \right\}$ if and only if it belongs to the family of models  $\left\{\begin{tikzpicture}[scale=.6, baseline=(current bounding box.center)]
\draw[thick,->](0,0)--(1,1);
\end{tikzpicture}, \begin{tikzpicture}[scale=.6, baseline=(current bounding box.center)]
\draw[thick,->](0,0)--(-1,-1);
\end{tikzpicture} \right\}$. The same holds for the anti-diagonal configuration.
If the model belongs to the family of models $\left\{\begin{tikzpicture}[scale=.6, baseline=(current bounding box.center)]
\draw[thick,->](0,0)--(1,1);
\end{tikzpicture}, \begin{tikzpicture}[scale=.6, baseline=(current bounding box.center)]
\draw[thick,->](0,0)--(-1,-1);
\end{tikzpicture} \right\}$, then the {\it kernel} $$
K(x,y,t)=-d_{-1,-1}t +xy-d_{0,0}txy -d_{1,1}t x^{2}y^{2}\in \C[xy]
$$
is a degree two polynomial in $xy$. Thus it may be factorized in the following form ${K(x,y,t)= -d_{1,1}t(xy-\alpha)(xy-\beta)}$ for some $\alpha,\beta\in \C$.
If the model belongs to the family of models $\left\{\begin{tikzpicture}[scale=.6, baseline=(current bounding box.center)]
\draw[thick,->](0,0)--(-1,1);
\end{tikzpicture}, \begin{tikzpicture}[scale=.6, baseline=(current bounding box.center)]
\draw[thick,->](0,0)--(1,-1);
\end{tikzpicture} \right\}$, then 
$$
K(x,y,t)=-d_{-1,1}ty^{2} +xy- d_{0,0}txy-d_{1,-1}tx^{2}.
$$ 
In this situation, $K(x,y,t)y^{-2}\in \C[x/y]$ may be factorized in the ring $\C[x/y]$, proving that $K(x,y,t)$ may be factorized in $\C[x,y]$ as well.
 
 Conversely, let us assume that {\it the model of the walk} is {\it degenerate}.  Recall that we have assumed that cases (\ref{case1}) and (\ref{case2}) do not hold, so $K(x,y,t)$ has $x$- and $y$-degree two, $x$- and $y$-valuation $0$, and is reducible. We have to prove that the  model belongs to  one of the family of models $\left\{\begin{tikzpicture}[scale=.6, baseline=(current bounding box.center)]
\draw[thick,->](0,0)--(1,1);
\end{tikzpicture}, \begin{tikzpicture}[scale=.6, baseline=(current bounding box.center)]
\draw[thick,->](0,0)--(-1,-1);
\end{tikzpicture} \right\}$ or  $\left\{\begin{tikzpicture}[scale=.6, baseline=(current bounding box.center)]
\draw[thick,->](0,0)--(-1,1);
\end{tikzpicture}, \begin{tikzpicture}[scale=.6, baseline=(current bounding box.center)]
\draw[thick,->](0,0)--(1,-1);
\end{tikzpicture} \right\}$.
Let us write  a factorization 
$$
K(x,y,t)=-f_{1}(x,y)f_{2}(x,y), 
$$ 
with $f_{1}(x,y),f_{2}(x,y)\in \C[x,y]$ not constant. Let us now prove several lemmas on the the polynomials $f_{1}(x,y),f_{2}(x,y)\in \C[x,y]$.  \par

\begin{lem}\label{lem1}
Both $f_{1}(x,y)$ and $f_{2}(x,y)$ have bidegree $(1,1)$.
\end{lem}
\begin{proof}[Proof of Lemma \ref{lem1}]
  Suppose to the contrary that $f_{1}(x,y)$ or $f_{2}(x,y)$ does not  have bidegree $(1,1)$. Since $K$ is
 of bidegree at most $(2,2)$ then at least one of the $f_i$'s has degree $0$ in $x$ or $y$.
Up to interchange of  $x$ and $y$ and $f_1$ and $f_2$, we may assume that $f_{1}(x,y)$ has $y$-degree $0$ and we denote it by $f_{1}(x)$. Since we are not in Cases (\ref{case1}) and (\ref{case2}) of Proposition \ref{prop:singcases}, the  polynomials ${d_{-1,-1}t+d_{0,-1}tx+d_{1,-1}tx^{2}}$ and ${d_{-1,0}t+(d_{0,0}t-1)x+d_{1,0}tx^{2}}$ are nonzero.
By ${K(x,y,t)=-f_{1}(x)f_{2}(x,y)}$, we find in particular that $f_{1}(x)$ is a common factor of the nonzero polynomials ${d_{-1,-1}t+d_{0,-1}tx+d_{1,-1}tx^{2}}$ and ${d_{-1,0}t+(d_{0,0}t-1)x+d_{1,0}tx^{2}}$. 
Since $t$ is nonzero, we find that the roots of 
${d_{-1,-1}t+d_{0,-1}tx+d_{1,-1}tx^{2}=0}$ are algebraic over $\Q(d_{i,j})$. On the other hand, since $t$ is transcendental over $\Q(d_{i,j})$, if $x$ is a root of ${d_{-1,0}t+(d_{0,0}t-1)x+d_{1,0}tx^{2}=0}$ that is algebraic over $\Q(d_{i,j})$, then the constant term in $t$ has to be zero, proving that $x=0$. Therefore, they are  polynomials with only zero as a potential common roots. So the only potential root of $f_1(x)$ is zero. This means that either $f_{1}(x)$ has degree $0$,  i.e. $f_{1}(x)\in \C$, or $x$ divides $f_1(x)$. In the latter case, $x$ divides  $K(x,y,t)$, and we are in Case \ref{case1}. In both cases, this is a contradiction and proves the lemma.  
\end{proof}

\begin{lem}\label{lem2}
The polynomials $f_1(x,y)$ and $f_2(x,y)$ are irreducible in the ring $\C[x,y]$.
\end{lem}
\begin{proof}[Proof of Lemma \ref{lem2}]
To the contrary, suppose that we can find a factorization ${f_1(x,y) = (ax-b)(cy-d)}$ for some $ a,b,c,d \in \C$. Since $f_{1}(x,y)$ has bidegree $(1,1)$, we have $ac \neq 0$. We then have that 
 $$0 = K(b/a,y,t) =\frac{b}{a}y - t(\tilde{A}_{-1}(\frac{b}{a}) + \tilde{A}_0(\frac{b}{a})y + \tilde{A}_1(\frac{b}{a})y^2) $$
 where $\tilde{A}_i = xA_i \in \Q[x]$. Note that $\tilde{A}_1(x)$ is nonzero because $K(x,y,t)$ has bidegree $(2,2)$. Equating the $y^{2}$-terms we find that $\tilde{A}_1(\frac{b}{a})=0$ so $\frac{b}{a}$ is algebraic over $\Q(d_{i,j})$. Equating the $y$-terms, we obtain that $\frac{b}{a} -t \tilde{A}_0(\frac{b}{a}) =0$. Using  the fact that $t$ is transcendental over $\Q(d_{i,j})$ and $\frac{b}{a}$ is algebraic over $\Q(d_{i,j})$, we deduce $\frac{b}{a} = 0$. Therefore $b=0$. This contradicts the fact that $K$ has  $x$-valuation $0$. A similar argument shows that $f_2(x,y)$ is irreducible.  
 \end{proof}

\begin{lem}\label{lem3}
Let $\overline{f}_i(x,y)$ denote the polynomial whose coefficients are the complex conjugates of those of $f_i(x,y)$. We may reduce to the case where one of the following cases hold:
 \begin{itemize}
 \item there exists $\epsilon \in \{\pm 1\}$ such that $\overline{f_1}(x,y)=\epsilon f_{2}(x,y)$,
 \item $\overline{f_1}(x,y)=f_{1}(x,y) \in \RX[x,y] $ and  $\overline{f_2}(x,y)=f_{2}(x,y)\in \RX[x,y].$
\end{itemize}
 \end{lem}
 
 \begin{proof}[Proof of Lemma \ref{lem3}]
  Unique factorization of polynomials implies that since ${{-K(x,y,t)}=f_{1}(x,y)f_{2}(x,y)=\overline{f_1}(x,y)\overline{f_2}(x,y)}$, there exists $\lambda \in \C^{*}$ such that
\begin{itemize}
\item either $\overline{f_1}(x,y)=\lambda f_{2}(x,y)$ and  $\overline{f_2}(x,y)=\lambda^{-1} f_{1}(x,y)$;
 \item or $\overline{f_1}(x,y)=\lambda f_{1}(x,y)$ and  $\overline{f}_2(x,y)=\lambda^{-1} f_{2}(x,y)$.
\end{itemize}

In the former case, we  have $f_1(x,y) = \overline{\lambda}\,\overline{f_2}(x,y) = \overline{\lambda}\lambda^{-1} f_{1}(x,y)$ and so ${\overline{\lambda}\lambda^{-1}=1}$. This implies that $\lambda$ is real and replacing $f_{1}(x,y)$ by $\vert \lambda \vert^{-1/2}f_{1}(x,y)$ and $f_{2}(x,y)$ by $\vert \lambda \vert^{1/2}f_{2}(x,y)$, we can assume that either $\overline{f_1}(x,y)=f_{2}(x,y)$ and  $\overline{f_2}(x,y)=f_{1}(x,y)$ or $\overline{f_1}(x,y)=-f_{2}(x,y)$ and  $\overline{f_2}(x,y)=-f_{1}(x,y)$.\par

 A similar computation in the latter case shows that $|\lambda | = 1$. Letting $\mu $ be a square root of $\lambda$ we have $\mu^{-1} = \overline{\mu} $ so $\lambda = \mu/\overline{\mu}$. Replacing $f_{1}(x,y)$ by $\mu f_{1}(x,y)$ and $f_{2}(x,y)$ by $\overline{\mu}f_{2}(x,y)$, we can assume that $\overline{f_1}(x,y)=f_{1}(x,y)$ and  $\overline{f_2}(x,y)=f_{2}(x,y)$.  
\end{proof}

\vspace{.1in}

 Let us continue the proof of Proposition \ref{prop:singcases}.
For $i=1,2$, let us write 
$$f_{i}(x,y)=(\a_{i,4}x+\a_{i,3})y+(\a_{i,2}x+\a_{i,1}),$$
with $\a_{i,j}\in \C$. Equating the terms in $x^{i}y^{j}$ with $-1\leq i,j\leq 1$, in $f_{1}(x,y)f_{2}(x,y)=-K(x,y,t)$, we find (recall that $d_{i,j}\in [0,1]$, $t\in ]0,1[$)
$$\begin{array}{| l | l | l |}
\hline
\text{term}& \text{coefficient in }f_{1}(x,y)f_{2}(x,y) & \text{coefficient in } -K(x,y,t)\\\hline
1&\a_{1,1}\a_{2,1}&d_{-1,-1}t\geq 0\\
x&\a_{1,2}\a_{2,1}+\a_{1,1}\a_{2,2}&d_{0,-1}t\geq 0\\
x^{2}&\a_{1,2}\a_{2,2}&d_{1,-1}t\geq 0\\
y&\a_{1,3}\a_{2,1}+\a_{1,1}\a_{2,3}&d_{-1,0}t\geq 0\\
xy&\a_{1,4}\a_{2,1}+\a_{1,3}\a_{2,2}+\a_{1,2}\a_{2,3}+\a_{1,1}\a_{2,4}&d_{0,0}t-1< 0\\
x^{2}y&\a_{1,4}\a_{2,2}+\a_{1,2}\a_{2,4}&d_{1,0}t\geq 0\\
y^{2}&\a_{1,3}\a_{2,3}&d_{-1,1}t\geq 0\\
xy^{2}&\a_{1,4}\a_{2,3}+\a_{1,3}\a_{2,4}&d_{0,1}t\geq 0\\
x^{2}y^{2}&\a_{1,4}\a_{2,4}&d_{1,1}t\geq  0\\ \hline
\end{array}$$
Let us treat separately two cases.\\ \par

\noindent \textbf{Case 1: $f_{1}(x,y),f_{2}(x,y)\notin \R[{x,y}]$.}
 So, in this case we have either $\overline{f_1}(x,y)=f_{2}(x,y)$  or $\overline{f_1}(x,y)=-f_{2}(x,y)$ .

Let us first assume that $\overline{f_1}(x,y)=f_{2}(x,y)$. Then, evaluating the equality ${K(x,y,t)=-f_{1}(x,y)f_{2}(x,y)}$ at $x=y=1$, we get the following equality ${K(1,1,t)=-f_{1}(1,1)f_{2}(1,1)=-\vert f_{1}(1,1)\vert^{2}}$. But this is impossible because the left-hand term $K(1,1,t)=1-t\sum_{i,j\in \{-1,0,1\}^{2}}d_{i,j}=1-t$ is $> 0$ whereas the right-hand term $-\vert f_{1}(1,1)\vert^{2}$ is $\leq 0$. 
 
Let us now assume that $\overline{f_1}(x,y)=-f_{2}(x,y)$. Equating the constant terms in the equality $f_{1}(x,y)f_{2}(x,y)=-K(x,y,t)$, we get $-\vert \a_{1,1} \vert^{2}=d_{-1,-1}t$, so $\a_{1,1}=\a_{2,1}=d_{-1,-1}=0$. Equating the coefficients of $x^{2}$ in the equality $f_{1}(x,y)f_{2}(x,y)=-K(x,y,t)$, we get $-\vert \a_{1,2} \vert^{2}=d_{1,-1}t$, so $\a_{1,2}=\a_{2,2}=d_{1,-1}=0$. It follows that the $y$-valuation of $f_{1}(x,y)f_{2}(x,y)=-K(x,y,t)$ is $\geq 2$, whence a contradiction. \\\par
\noindent \textbf{Case 2: $f_{1}(x,y),f_{2}(x,y)\in \R[{x,y}]$.} We first claim that, after possibly replacing $f_{1}(x,y)$ by $-f_{1}(x,y)$ and $f_{2}(x,y)$ by $-f_{2}(x,y)$, we may assume that $\a_{1,4},\a_{2,4},\a_{1,3},\a_{2,3}\geq 0$.

Let us first assume that $\a_{1,4}\a_{2,4} \neq 0$. Since $\a_{1,4}\a_{2,4}=d_{1,1}t\geq 0$,  we find that $\a_{1,4},\a_{2,4}$ belong simultaneously to $\R_{> 0}$ or $\R_{< 0}$. After possibly replacing $f_{1}(x,y)$ by $-f_{1}(x,y)$ and $f_{2}(x,y)$ by $-f_{2}(x,y)$, we may assume that $\a_{1,4},\a_{2,4} > 0$. Since ${\a_{1,3}\a_{2,3}=d_{-1,1}t\geq 0}$, we have that $\a_{1,3},\a_{2,3}$ belong simultaneously to $\R_{\geq 0}$ or $\R_{\leq 0}$. Then, the equality $\a_{1,4}\a_{2,3}+\a_{1,3}\a_{2,4}=d_{0,1}t\geq 0$ implies that $\a_{1,3},\a_{2,3}\geq 0$.

We can argue similarly in the case $\a_{1,3}\a_{2,3} \neq 0$. 

It remains to consider the case $\a_{1,4}\a_{2,4}=\a_{1,3}\a_{2,3}=0$. After possibly replacing $f_{1}(x,y)$ by $-f_{1}(x,y)$ and $f_{2}(x,y)$ by $-f_{2}(x,y)$, we may assume that $\a_{1,4},\a_{2,4} \geq 0$. The case $\a_{1,4}=\a_{1,3}=0$ is impossible because, otherwise, we would have $d_{1,1}=d_{-1,1}=d_{0,1}=0$, which is excluded. Similarly, the case $\a_{2,4}=\a_{2,3}=0$ is impossible. So, we are left with the cases $\a_{1,4}=\a_{2,3}=0$ or $\a_{2,4}=\a_{1,3}=0$. In both cases, the equality  $\a_{1,4}\a_{2,3}+\a_{1,3}\a_{2,4}=d_{0,1}t\geq 0$ with  $\a_{1,4},\a_{2,4} \geq 0$, implies that $\a_{1,4},\a_{2,4},\a_{1,3},\a_{2,3}\geq 0$. \par 
Arguing as above, we see that $\a_{1,2},\a_{2,2},\a_{1,1},\a_{2,1}$ all belong to $\R_{\geq 0}$ or  all belong to  $\R_{\leq 0}$.  
Using the equation of the $xy$-coefficients, we find that   $\a_{1,2},\a_{2,2},\a_{1,1},\a_{2,1}$ are all in $\R_{\leq 0}$.\\ 
Now, equating the coefficients of $x^{2}y$ in the equality $f_{1}(x,y)f_{2}(x,y)=-K(x,y,t)$ we get $\a_{1,4}\a_{2,2}+\a_{1,2}\a_{2,4}=d_{1,0}t$. Using the fact that $\a_{1,4}\a_{2,2}, \a_{1,2}\a_{2,4} \leq 0$ and that $d_{1,0}t \geq 0$, we get $\a_{1,4}\a_{2,2}= \a_{1,2}\a_{2,4}=d_{1,0}=0$. Similarly, using the coefficients of $y$, we get $\a_{1,3}\a_{2,1}=\a_{1,1}\a_{2,3}=d_{-1,0}=0$. \\
So, we have 
$$
\a_{1,4}\a_{2,2}=\a_{1,2}\a_{2,4}=\a_{1,3}\a_{2,1}=\a_{1,1}\a_{2,3}=0.
$$
The fact that $K(x,y,t)$ has $x$- and $y$-degree $2$ and $x$- and $y$-valuation $0$ implies that, for any $i \in \{1,2\}$, none of the vectors $(\a_{i,4},\a_{i,3})$, $(\a_{i,2},\a_{i,1})$, $(\a_{i,4},\a_{i,2})$ and $(\a_{i,3},\a_{i,1})$ is $(0,0)$. 
Since $\a_{1,4}\a_{2,2}=0$, we have $\a_{1,4}=0$ or $\a_{2,2}=0$. If $\a_{1,4}=0$, from what precedes, we find 
$$\a_{1,4}=\a_{2,4}=\a_{2,1}=\a_{1,1}=0. $$
If $\a_{2,2}=0$ we obtain
$$\a_{2,2}=\a_{1,2}=\a_{1,3}=\a_{2,3}=0.$$
In the first case, the  model belongs to the family of models $\left\{\begin{tikzpicture}[scale=.6, baseline=(current bounding box.center)]
\draw[thick,->](0,0)--(-1,1);
\end{tikzpicture}, \begin{tikzpicture}[scale=.6, baseline=(current bounding box.center)]
\draw[thick,->](0,0)--(1,-1);
\end{tikzpicture} \right\}$. 
In the second case, we find that the model belongs to the family of models $\left\{\begin{tikzpicture}[scale=.6, baseline=(current bounding box.center)]
\draw[thick,->](0,0)--(1,1);
\end{tikzpicture}, \begin{tikzpicture}[scale=.6, baseline=(current bounding box.center)]
\draw[thick,->](0,0)--(-1,-1);
\end{tikzpicture} \right\}$.
This completes the proof.
\end{proof}

\begin{rem}
The fact $d_{i,j}\in [0,1]$ are probabilities is crucial in the proof of Proposition~\ref{prop:singcases}. We do not expect this result to be correct for general $d_{i,j}\in \C$.
\end{rem}

\begin{rem}
 The ``{\it degenerate  models of walks}'' are called ``{\it singular}'' by certain authors,  e.g., in \cite{fayolle1999random,FIM}. Note also that, in \cite{KurkRasch}, ``{\it singular walks}'' has a different meaning and  refers to  {\it models of walks} such that the associated {\it kernel} defines a genus zero curve. 
\end{rem}

\begin{rem}\label{rem5}
In \cite[Proposition 3]{dreyfus2017differential}, the authors show that Proposition \ref{prop:singcases} extends  mutatis mutandis to the case when $t \in ]0,1[$ is algebraic over $\Q(d_{i,j})$. Their proof relies on Proposition~\ref{prop:singcases}  and uses a continuity argument of the parameter $t$ to deduce that Proposition~\ref{prop:singcases} stays correct for general values of $t\in ]0,1[$.
\end{rem}

From now on, we will only consider {\it nondegenerate models of walks}. In terms of  {\it models of walks}, this only discards one dimensional problems and  {\it models of walks} in the half-plane restricted to the quarter plane that are easier  to study, as  explained in \cite[Section 2.1]{BMM}. 

\section{Singularities and genus of the kernel curve}\label{sec:algcurvedefkernel}

The {\it kernel curve} $E_{t}$ is the complex affine algebraic curve defined by
 $$
 E_t = \{(x,y) \in \C \times \C \ \vert \ K(x,y,t) = 0\}.
 $$ 
We shall now consider a compactification of this curve. We let $\P1(\C)$ be the complex projective line, which is the quotient of $(\C \times \C) \setminus \{(0,0)\}$ by the equivalence relation $\sim$ defined by 
$$
(x_{0},x_{1}) \sim (x_{0}',x_{1}') \Leftrightarrow \exists \lambda \in \C^{*},  (x_{0}',x_{1}') = \lambda (x_{0},x_{1}). 
$$
The equivalence class of $(x_{0},x_{1}) \in (\C \times \C) \setminus \{(0,0)\}$ is denoted by $[x_{0}:x_{1}] \in \P1(\C)$. The map 
$
x \mapsto  [x:1]
$ 
embeds $\C$ inside $\P1(\C)$. The latter map is not surjective: its image is $\P1(\C) \setminus \{[1:0]\}$; the missing point $[1:0]$  is usually denoted by $\infty$. 
  Now, we embed $E_{t}$  inside $\P1(\C) \times \P1(\C)$ via  ${(x,y) \mapsto ([x:1],[y:1])}$. The {\it kernel curve} $\Etproj$ is the closure of this embedding of $E_{t}$.  In other words, the {\it kernel curve} $\Etproj $ is the algebraic curve defined by
$$
\Etproj = \{([x_{0}:x_{1}],[y_{0}:y_{1}]) \in \P1(\C) \times \P1(\C) \ \vert \ \overline{K}(x_0,x_1,y_0,y_1,t) = 0\}
$$
where $\overline{K}(x_0,x_1,y_0,y_1,t)$ is the following bihomogeneous polynomial
\begin{equation}\label{eq:kernelwalk}
\overline{K}(x_0,x_1,y_0,y_1,t)={x_1^2y_1^2K(\frac{x_0}{x_1},\frac{y_0}{y_1},t)}= x_0x_1y_0y_1 -t \sum_{i,j=0}^2 d_{i-1,j-1} x_0^{i} x_1^{2-i}y_0^j y_1^{2-j}. 
 \end{equation}

 Although it may seem more natural to take the closure of $\Etproj$ in $\mathbb{P}^2(\C)$, the above definition allows us to extend the involutions of $E_t$ of Section~\ref{appendix:involutions} in a natural way as well as allowing us to avoid unnecessary singularities.

We shall now study the singularities and compute the genus of $\Etproj$.
Recall that since the {\it model of walk} under consideration is {\it nondegenerate}, the polynomial $K(x,y,t)$ is irreducible.
We recall that by definition $P=([a:b],[c:d]) \in \Etproj$ is called a singularity of the irreducible kernel  $\Etproj$ if 
$$\frac{\partial \overline{K}(a,b,c,d,t)}{\partial x_{0}} = 
\frac{\partial \overline{K}(a,b,c,d,t)}{\partial x_{1}}
=\frac{\partial \overline{K}(a,b,c,d,t)}{\partial y_{0}}
=\frac{\partial \overline{K}(a,b,c,d,t)}{\partial y_{1}}=0. 
$$
 Note that one can check this condition in any affine neighborhood of a point. For example,  if $b,d \neq 0$, the bihomogeneity of $\overline{K}$ implies
\begin{align*}
  0= 2 \overline{K}(a/b,1,c/d,1,t )&= \frac{a}{b}\frac{\partial \overline{K}(a/b,1,c,/d,1,t)}{\partial x_{0}} +\frac{\partial \overline{K}(a/b,1,c,/d,1,t)}{\partial x_{1}}\\
 &= \frac{c}{d}\frac{\partial \overline{K}(a/b,1,c,/d,1,t)}{\partial y_{0}} +\frac{\partial \overline{K}(a/b,1,c,/d,1,t)}{\partial y_{1}}.
\end{align*}
Therefore the point  $P =([a:b],[c:d]) \in \Etproj$ is a singular point if and only if 
\begin{equation*}
     \frac{\partial \overline{K}(a/b,1,c/d,1,t)}{\partial x_{0}} =  \frac{\partial \overline{K}(a/b,1,c/d,1,t)}{\partial y_{0}}=0.
\end{equation*}

If $P=([a:b],[c:d]) \in \Etproj$ is not a singularity of $\Etproj$, then it is called a smooth point of $\Etproj$.

We also recall that $\Etproj$ is called singular if it has at least one singular point. Otherwise, we say that $\Etproj$ is nonsingular or smooth.

The genus of an algebraic curve is a classical notion in algebraic geometry. It is a nonnegative integer that we may attach to a curve, see \cite[Section~8.3]{fulton1984introduction} for a definition. The study of the genus of $\Etproj$ has been considered in \cite{FIM}.
 Proposition \ref{prop:genuszeroKernel} below shows that the smoothness of $\Etproj$ is intimately related to the value of the genus of $\Etproj$. 
  We define the genus of the {\it weighted model of walk}, as the genus of its {\it kernel curve} $\Etproj$. \\

Remind the following notations from the introduction~:
$$
\begin{array}{lll}
K(x,y,t)&=&xy-txA_{-1}(x)-txA_{0}(x)y-txA_{1}(x)y^{2},\\
&=&xy-tyB_{-1}(y)-tyB_{0}(y)x-tyB _{1}(y)x^{2},
\end{array}
 $$
 where $xA_{i}(x)\in \Q[x]$ and  $yB_{i}(y)\in \Q[y]$. Then we may write 
 $$
\begin{array}{lll}
\overline{K}(x_{0},x_{1},y_{0},y_{1},t)&=&\overline{C}_{1}(x_0,x_1,t)y_{1}^{2}+\overline{B}_{1}(x_0,x_1,t)y_{0}y_{1}+\overline{A}_{1}(x_0,x_1,t)y_{0}^{2}\\
&=&\overline{C}_{2}(y_0,y_1,t)x_{1}^{2}+\overline{B}_{2}(y_0,y_1,t)x_{0}x_{1}+\overline{A}_{2}(y_0,y_1,t)x_{0}^{2}.
\end{array}
 $$ 
 For any $[x_0:x_1]$ and $[y_0:y_1]$ in $\P1(\C)$, we denote by $\Delta_{1}([x_0:x_1])$ and $\Delta_{2}([y_0:y_1])$ the discriminants of the degree two homogeneous polynomials given by $y \mapsto \overline{K}(x_0,x_1,y,t)$ and  $x \mapsto \overline{K}(x,y_0,y_1,t)$ respectively, i.e.
\begin{multline}
\Delta_{1}([x_0:x_1])=\overline{B}_{1}(x_0,x_1,t)^{2}-4\overline{A}_{1}(x_0,x_1,t)\overline{C}_{1}(x_0,x_1,t) \\
=(x_{0}x_{1}-
t^2 \Big( (d_{-1,0} x_1^2 -  \frac{1}{t} x_0x_1 +d_{0,0}x_0x_1 + d_{1,0}x_0^2)^2  \\
 - 4(d_{-1,1} x_1^2 + d_{0,1} x_0x_1 + d_{1,1}x_0^2)(d_{-1,-1} x_1^2 + d_{0,-1} x_0x_1 + d_{1,-1}x_0^2) \Big) \nonumber
\end{multline}
and 
\begin{multline}
\Delta_{2}([y_0:y_1])=\overline{B}_{2}(y_0,y_1,t)^{2}-4\overline{A}_{2}(y_0,y_1,t)\overline{C}_{2}(y_0,y_1,t) \\
=t^2 \Big( (d_{0,-1} y_1^2 -  \frac{1}{t} y_0y_1 +d_{0,0}y_0y_1+ d_{0,1}y_0^2)^2  \\ 
- 4(d_{1,-1} y_1^2 + d_{1,0} y_0y_1 + d_{1,1}y_0^2)(d_{-1,-1} y_1^2 + d_{-1,0} y_0y_1 + d_{-1,1}y_0^2) \Big).\nonumber
\end{multline}

As we will see in the sequel, see Remark \ref{rem3}, $\Delta_{1}([x_0:x_1])$ has a double root if and only if $\Delta_{2}([y_0:y_1])$ has a double root.
\begin{prop}\label{prop:genuszeroKernel}
For nondegenerate models, the following facts are equivalent:
\begin{enumerate}
\item the curve $\Etproj$ is a genus zero curve;
\item the curve $\Etproj$ is singular;
\item the curve $\Etproj$ has exactly one singularity $\Omega\in \Etproj$;
\item there exists $([a:b],[c:d])\in \Etproj$ such that the discriminants $\Delta_{1}([x_0:x_1])$ and $\Delta_{2}([y_0:y_1])$ have a root $[a:b]\in \P1(\C)$ and $[c:d]\in \P1(\C)$ respectively;
 \item there exists $([a:b],[c:d])\in \Etproj$ such that the discriminants $\Delta_{1}([x_0:x_1])$ and $\Delta_{2}([y_0:y_1])$ have a double root $[a:b]\in \P1(\C)$ and $[c:d]\in \P1(\C)$ respectively.
\end{enumerate}
If these properties are satisfied, then the singular point is $\Omega=([a:b],[c:d])$ where $[a:b] \in \P1(\C)$ is a double root of $\Delta_{1}([x_0:x_1])$ and $[c:d] \in \P1(\C)$ is a double root of $\Delta_{2}([y_0:y_1])$. 
If the previous properties are not satisfied, then $\Etproj$ is a smooth curve of genus one.
\end{prop}
\begin{proof}
By \cite[Section 3.3.1]{DuistQRT}, the following formula gives  the genus of $\Etproj$,
\begin{equation}\label{equn:virtualgenusproof}
g(\Etproj)= 1 -\sum_{P \in \mathrm{Sing}}  \frac{m(P)(m(P)-1)}{2},
\end{equation} 
where $m(P)$ is a positive  integer  standing for the multiplicity of a  point $P$, that is, some partial derivative of order  $m(P)$ does not vanish while for every $\ell < m(P)$, the partial derivatives of order $\ell$  vanish at $P$. Since the genus is a nonnegative integer, the above formula shows that $g(\Etproj)$ is equal to $0$ or $1$. This formula shows more precisely that $\Etproj$  is  smooth  if and only if $g(\Etproj)= 1$.
 Moreover \eqref{equn:virtualgenusproof} shows that if $\Etproj$ is singular, then there is exactly  one singular point that is a double point, and the curve has genus zero. This proves the equivalence between (1), (2) and (3), and the last statement of the proposition.\par 

Let us prove (4) $\Rightarrow$ (3). Assume that  the discriminant $\Delta_{1}([x_0:x_1])$ (resp. $\Delta_{2}([y_0:y_1])$) has a  root in $[a:b]\in \P1(\C)$ (resp. $[c:d]\in \P1(\C)$).  Let us write

$$\begin{array}{llll}
&\overline{K}(x_0,x_1,y_0,y_1,t)\\
=& e_{-1,1} (dy_0-cy_1)^{2} &+ e_{0,1}(b x_0-a x_1)(dy_0-cy_1)^2 &+ e_{1,1}(b x_0-a x_1)^{2}(dy_0-cy_1)^2  \\
+&e_{-1,0} (dy_0-cy_1) &+ e_{0,0}(b x_0-a x_1)(dy_0-cy_1) &+ e_{1,0}(b x_0-a x_1)^{2}(dy_0-cy_1)\\
+&{e_{-1,-1}}&  +e_{0,-1}(b x_0-a x_1)&+ e_{1,-1}(b x_0-a x_1)^{2}.
\end{array}$$
{Since $([a:b],[c:d])\in \Etproj$, we have by definition that $\overline{K}(a,b,c,d,t)=0$,  i.e. $e_{-1,-1}=0$.}
Since $\Delta_{1}([x_0:x_1])$ has a  root in $[a:b]\in \P1(\C)$,   $K(a,b,y_0,y_1)$ has a double root at $ [c,d]$ and so $e_{-1,0}=0$. Similarly,  the fact that $\Delta_{2}([y_0:y_1])$ has a  root in $[c:d]\in \P1(\C)$ implies $e_{0,-1}=0$. This shows that 
{
$$\frac{\partial \overline{K}(a,b,c,d,t)}{\partial x_{0}} = 
\frac{\partial \overline{K}(a,b,c,d,t)}{\partial x_{1}}
=\frac{\partial \overline{K}(a,b,c,d,t)}{\partial y_{0}}
=\frac{\partial \overline{K}(a,b,c,d,t)}{\partial y_{1}}=0, 
$$}
and, hence, $([a:b],[c:d])$ is the singular point of $ \Etproj$. \par 
Let us prove (3) $\Rightarrow$ (5). If $\Omega=([a:b],[c:d])$ is the singular point of $ \Etproj$, then  $e_{-1,-1}=e_{-1,0}=e_{0,-1}=0$, and the discriminants $\Delta_{1}([x_0:x_1])$ and $\Delta_{2}([y_0:y_1])$ have a double root in $[a:b]\in \P1(\C)$ and $[c:d]\in \P1(\C)$ respectively. 

The implication (5) $\Rightarrow$ (4) is obvious.
\end{proof}

Our next aim is to describe the genus zero {\it models of walks}.

\begin{lem}\label{theo:doublezero} The discriminant $\Delta_{2}([y_0:y_1])$ has a double zero if and only if the model of the walk is included in a closed half plane whose boundary passes through $(0, 0)$.
In other word, this correspond to {\it models of the walks} that belong to one of the following eight families
$$\begin{tikzpicture}[scale=.6, baseline=(current bounding box.center)]
\foreach \x in {-1,0,1} \foreach \y in {-1,0,1} \fill(\x,\y) circle[radius=0pt];
\draw[dashed,->,line width=0.5mm](0,0)--(0,1);
\draw[dashed,->,line width=0.5mm](0,0)--(1,1);
\draw[dashed,->,line width=0.5mm](0,0)--(1,0);
\draw[dashed,->,line width=0.5mm](0,0)--(1,-1);
\draw[dashed,->,line width=0.5mm](0,0)--(0,-1);
\end{tikzpicture}  \quad \begin{tikzpicture}[scale=.6, baseline=(current bounding box.center)]
\foreach \x in {-1,0,1} \foreach \y in {-1,0,1} \fill(\x,\y) circle[radius=0pt];
\draw[dashed,->,line width=0.5mm](0,0)--(-1,1);
\draw[dashed,->,line width=0.5mm](0,0)--(0,1);
\draw[dashed,->,line width=0.5mm](0,0)--(1,1);
\draw[dashed,->,line width=0.5mm](0,0)--(1,0);
\draw[dashed,->,line width=0.5mm](0,0)--(1,-1);
\end{tikzpicture}\quad 
\begin{tikzpicture}[scale=.6, baseline=(current bounding box.center)]
\foreach \x in {-1,0,1} \foreach \y in {-1,0,1} \fill(\x,\y) circle[radius=0pt];
\draw[dashed,->,line width=0.5mm](0,0)--(-1,1);
\draw[dashed,->,line width=0.5mm](0,0)--(0,1);
\draw[dashed,->,line width=0.5mm](0,0)--(1,1);
\draw[dashed,->,line width=0.5mm](0,0)--(-1,0);
\draw[dashed,->,line width=0.5mm](0,0)--(1,0);
\end{tikzpicture}\quad 
\begin{tikzpicture}[scale=.6, baseline=(current bounding box.center)]
\foreach \x in {-1,0,1} \foreach \y in {-1,0,1} \fill(\x,\y) circle[radius=0pt];
\draw[dashed,->,line width=0.5mm](0,0)--(-1,1);
\draw[dashed,->,line width=0.5mm](0,0)--(1,1);
\draw[dashed,->,line width=0.5mm](0,0)--(-1,0);
\draw[dashed,->,line width=0.5mm](0,0)--(0,1);
\draw[dashed,->,line width=0.5mm](0,0)--(-1,-1);
\end{tikzpicture}\quad\begin{tikzpicture}[scale=.6, baseline=(current bounding box.center)]
\foreach \x in {-1,0,1} \foreach \y in {-1,0,1} \fill(\x,\y) circle[radius=0pt];
\draw[dashed,->,line width=0.5mm](0,0)--(-1,1);
\draw[dashed,->,line width=0.5mm](0,0)--(0,1);
\draw[dashed,->,line width=0.5mm](0,0)--(-1,0);
\draw[dashed,->,line width=0.5mm](0,0)--(-1,-1);
\draw[dashed,->,line width=0.5mm](0,0)--(0,-1);
\end{tikzpicture}\quad\begin{tikzpicture}[scale=.6, baseline=(current bounding box.center)]
\foreach \x in {-1,0,1} \foreach \y in {-1,0,1} \fill(\x,\y) circle[radius=0pt];
\draw[dashed,->,line width=0.5mm](0,0)--(-1,1);
\draw[dashed,->,line width=0.5mm](0,0)--(-1,0);
\draw[dashed,->,line width=0.5mm](0,0)--(-1,-1);
\draw[dashed,->,line width=0.5mm](0,0)--(0,-1);
\draw[dashed,->,line width=0.5mm](0,0)--(1,-1);
\end{tikzpicture}\quad\begin{tikzpicture}[scale=.6, baseline=(current bounding box.center)]
\foreach \x in {-1,0,1} \foreach \y in {-1,0,1} \fill(\x,\y) circle[radius=0pt];
\draw[dashed,->,line width=0.5mm](0,0)--(-1,0);
\draw[dashed,->,line width=0.5mm](0,0)--(1,0);
\draw[dashed,->,line width=0.5mm](0,0)--(-1,-1);
\draw[dashed,->,line width=0.5mm](0,0)--(0,-1);
\draw[dashed,->,line width=0.5mm](0,0)--(1,-1);
\end{tikzpicture}\quad\begin{tikzpicture}[scale=.6, baseline=(current bounding box.center)]
\foreach \x in {-1,0,1} \foreach \y in {-1,0,1} \fill(\x,\y) circle[radius=0pt];
\draw[dashed,->,line width=0.5mm](0,0)--(1,1);
\draw[dashed,->,line width=0.5mm](0,0)--(1,0);
\draw[dashed,->,line width=0.5mm](0,0)--(-1,-1);
\draw[dashed,->,line width=0.5mm](0,0)--(0,-1);
\draw[dashed,->,line width=0.5mm](0,0)--(1,-1);
\end{tikzpicture}\quad$$
\end{lem}

\begin{rmk}\label{rem3}
As the statement of Lemma \ref{theo:doublezero} is symmetric with respect to $x$ and $y$ we deduce that the same holds for $\Delta_{1}([x_0:x_1])$. We then deduce that $\Delta_{1}([x_0:x_1])$ has a double root   if and only if $\Delta_{2}([y_0:y_1])$ has a double root.
\end{rmk}

\begin{rmk}
 In the case $t=1$, it is proved in \cite[Lemma 2.3.10]{FIM} that, besides the models listed in Lemma \ref{theo:doublezero}, any {\it nondegenerate} model such that the drift is zero, i.e.
\begin{equation*}
     \textstyle(\sum_{i}id_{i,j},\sum_{j}j d_{i,j})=(0,0),
\end{equation*}
has a curve $\Etproj$ of genus $0$. 
\end{rmk}

\begin{proof}
The computations seem to be too complicated to be performed by hand, so we used {\sc maple}\footnote{The maple worksheet is available at \url{https://singer.math.ncsu.edu/ms_papers.html}.}. We are going to prove the result with two strategies. This first one is to write the discriminant of the discriminant $\Delta_{2}([y_0:y_1])$ and study when the latter is $0$. The second strategy consists in decomposing the radical of an ideal into its prime components.\\ 
Let us first consider the situation where the double root is at $(a, b)$ where $b$ is not zero. Let us set $y_1 = 1$ and $y_0 = y$ to obtain the specialization $\Delta_{2}([y:1])$ of ${\Delta_{2}([y_0:y_1])}$.

 The following {\sc Maple} code calculates the discriminant of the discriminant, its degree  and order of vanishing in $t$, and then sets the coefficients of powers of $t$ equal to zero.  Solving these equations yields the $8$ solutions {\bf S[i]}, i = 1, \ldots ,8 corresponding to the $8$ stepsets listed in Lemma~\ref{theo:doublezero}.

\begin{verbatim}> K := expand(x*y*(1-t*(add(add(d[i, j]*x^i*y^j, i = -1 .. 1), j = -1 .. 1)))):
> DX :=  discrim(K, x):
> DD := discrim(discrim(K,x),y);
> ldegree(DD,t); degree(DD,t);
\end{verbatim}
\begin{center}{\small $4$\\ $12$}\end{center}
\begin{verbatim}> S := solve({seq(coeff(DD,t,i),i=4..12)},[seq(seq(d[i,j],i=-1..1),j=-1..1)]);
> nops(S);
\end{verbatim}
\begin{center}{\small$8$}\end{center}
\begin{verbatim}
> S[1];S[2];S[3];S[4];S[5];S[6];S[7];S[8];
\end{verbatim}

 An alternate approach is to use the {\it PolynomialIdeals} package

\begin{verbatim}> with(PolynomialIdeals):\end{verbatim}

and consider the prime decomposition of the radical of the ideal

\begin{verbatim}> J := <seq(coeff(DD,t,i), i=4..12)>:
> PrimeDecomposition(J);\end{verbatim}

$$\begin{array}{lll}
<d_{-1, -1}, d_{-1, 0}, d_{-1, 1}>,& \  <d_{-1, -1}, d_{-1, 0}, d_{0, -1}>,& \ <d_{-1, -1}, d_{0, -1}, d_{1, -1}>,\\
 \ <d_{-1, 0}, d_{-1, 1}, d_{0, 1}>, &
<d_{-1, 1}, d_{0, 1}, d_{1, 1}>,&\ <d_{0, -1}, d_{1, -1}, d_{1, 0}>, \\
\ <d_{0, 1}, d_{1, 0}, d_{1, 1}>,&\ <d_{1, -1}, d_{1, 0}, d_{1, 1}>.&\end{array}$$
\vspace{.1in}

The {\it PrimeDecomposition} command lists a set of prime ideals whose intersection is the radical of the original ideal. In particular, these ideals have the property that any zero of the original ideal is a zero of one of the listed ideals and vice versa, see \cite[Chapter 4, Section 6]{CLO97}. These again correspond to the eight step sets listed in Lemma~\ref{theo:doublezero}.\\

 We now consider the case where  the double root is at $(a, b)$ where $b=0$, that is, at $(1,0)$. 
 We will show that this case leads to {\it models of walks}  already mentioned above. 
\begin{verbatim} > DDX := expand(z^4*subs(y = 1/z, DX)):\end{verbatim}
If $z=0$ is a double root then the coefficient of $1$ and $z$ must be zero
 \begin{verbatim} > coeff(DDX, z, 0); coeff(DDX, z, 1); \end{verbatim}
\hspace*{.3in}\small$-4\,{t}^{2}d_{{-1,1}}d_{{1,1}}+{t}^{2}{d_{{0,1}}}^{2}$\\[0.1in]
\hspace*{.5in}\small$-4\,{t}^{2}d_{{-1,0}}d_{{1,1}}-4\,{t}^{2}d_{{-1,1}}d_{{1,0}}+2\,{t}^{2
}d_{{0,0}}d_{{0,1}}-2\,td_{{0,1}}$
\vspace{.1in}

Taking into account that $t$ is transcendental over $\Q(d_{i,j})$, we are led to three cases, corresponding to three of the step sets in Lemma~\ref{theo:doublezero}.\\

 \hspace*{.3in}\small$[d_{{0,1}}=0,d_{{-1,1}}=0,d_{{-1,0}}=0]$\\
\hspace*{.5in}\small$[d_{{0,1}}=0,d_{{-1,1}}=0,d_{{1,1}}=0]$\\
\hspace*{.5in}\small$[d_{{0,1}}=0,d_{{1,1}}=0,d_{{1,0}}=0]$ \end{proof}

 \begin{rem} The proof of Proposition~\ref{prop:singcases} proceeds by a direct ``hand calculation'' while the proof of Lemma~\ref{theo:doublezero} follows from a simple {\sc maple} calculation.  It would be interesting to have a simple {\sc maple} based proof of Proposition~\ref{prop:singcases} and a hand calculation proof of Lemma~\ref{theo:doublezero}.
\end{rem}

\begin{cor}\label{cor1}
The following holds:
\begin{enumerate}
\item
The {\it nondegenerate} genus zero  {\it models of walks} are the {\it nondegenerate models} whose  step set is included in an half space whose boundary passes through $(0, 0)$. More precisely, they
are {\it nondegenerate models} belonging to one of the following families
\begin{equation}\label{the five step set}\tag{G0}
\begin{tikzpicture}[scale=.6, baseline=(current bounding box.center)]
\foreach \x in {-1,0,1} \foreach \y in {-1,0,1} \fill(\x,\y) circle[radius=0pt];
\draw[thick,->](0,0)--(-1,1);
\draw[thick,->](0,0)--(1,-1);
\draw[dashed,->,line width=0.5mm](0,0)--(1,1);
\draw[dashed,->,line width=0.5mm](0,0)--(1,0);
\draw[dashed,->,line width=0.5mm](0,0)--(0,1);
\end{tikzpicture}
\quad 
\begin{tikzpicture}[scale=.6, baseline=(current bounding box.center)]
\foreach \x in {-1,0,1} \foreach \y in {-1,0,1} \fill(\x,\y) circle[radius=0pt];
\draw[thick,->](0,0)--(1,1);
\draw[thick,->](0,0)--(-1,-1);
\draw[dashed,->,line width=0.5mm](0,0)--(1,0);
\draw[dashed,->,line width=0.5mm](0,0)--(0,-1);
\draw[dashed,->,line width=0.5mm](0,0)--(1,-1);
\end{tikzpicture}\quad 
\begin{tikzpicture}[scale=.6, baseline=(current bounding box.center)]
\foreach \x in {-1,0,1} \foreach \y in {-1,0,1} \fill(\x,\y) circle[radius=0pt];
\draw[thick,->](0,0)--(-1,1);
\draw[thick,->](0,0)--(1,-1);
\draw[dashed,->,line width=0.5mm](0,0)--(-1,-1);
\draw[dashed,->,line width=0.5mm](0,0)--(-1,0);
\draw[dashed,->,line width=0.5mm](0,0)--(0,-1);
\end{tikzpicture}
\quad 
\begin{tikzpicture}[scale=.6, baseline=(current bounding box.center)]
\foreach \x in {-1,0,1} \foreach \y in {-1,0,1} \fill(\x,\y) circle[radius=0pt];
\draw[thick,->](0,0)--(1,1);
\draw[thick,->](0,0)--(-1,-1);
\draw[dashed,->,line width=0.5mm](0,0)--(-1,0);
\draw[dashed,->,line width=0.5mm](0,0)--(0,1);
\draw[dashed,->,line width=0.5mm](0,0)--(-1,1);
\end{tikzpicture}
\end{equation}
\item The {\it nondegenerate} genus one  {\it models of walks} are the models whose  step set is not included in any half space whose boundary passes through $(0, 0)$. 
\end{enumerate}
\end{cor}

\begin{rmk}\label{rem4}
See also \cite[Proposition 9]{dreyfus2017differential} for an extension of Corollary \ref{cor1} to the case when $t\in ]0,1[$ is algebraic over $\Q(d_{i,j})$.  Their proof relies on the results of the present section  and uses a continuity argument with respect to the parameter $t$ to deduce that Corollary \ref{cor1} stays correct for general values of $t\in ]0,1[$.
\end{rmk}

\begin{proof}
We use Proposition \ref{prop:genuszeroKernel}. We have to determine when there exists  ${([a:b],[c:d])\in \Etproj}$ such that the discriminants $\Delta_{1}([x_0:x_1])$ and $\Delta_{2}([y_0:y_1])$ have a double root $[a:b]\in \P1(\C)$ and $[c:d]\in \P1(\C)$. Lemma \ref{theo:doublezero} provides such  configurations. By Proposition \ref{prop:singcases}, the configurations number $1,3,5,7$ are dismissed since they led to {\it singular walks}. Then, if we are considering  {\it nondegenerate} genus zero  {\it models of walks}, we are in the four families of models considered in \eqref{the five step set}.  Furthermore, if the  step set is not included in any half space whose boundary passes through $(0, 0)$, the configurations of Proposition \ref{prop:singcases} and Lemma \ref{theo:doublezero} are excluded and by Proposition~\ref{prop:genuszeroKernel}, we are in the genus $1$ situation.\\
  Conversely, it remains to prove that if the {\it models of walks} are in the four families of {\it models} considered in \eqref{the five step set}, the kernel has genus $0$. Thanks to Proposition~\ref{prop:genuszeroKernel} it suffices to prove that the discriminants have a common zero in that case. This is the goal of the  following Lemma \ref{lem:disczeroes} and Remark \ref{rem1}. 
\end{proof}

 Let us write 
$$\Delta_{1}([x_0:x_1])= \displaystyle\sum_{i=0}^{4}\alpha_{i}(t)x_{0}^{i}x_{1}^{4-i}, \quad
\hbox{ and }
\Delta_{2}([y_0:y_1])= \displaystyle\sum_{i=0}^{4}\beta_{i}(t)y_{0}^{i}y_{1}^{4-i}.$$
where 
$$\begin{array}{lll}
\a_{0}(t)&=&t^{2}d_{-1,0}^{2}-4t^{2}d_{-1,1}d_{-1,-1}\\
\a_{1}(t)&=&2t^{2}d_{-1,0}d_{0,0}-2td_{-1,0}-4t^{2}d_{-1,1}d_{0,-1}-4t^{2}d_{0,1}d_{-1,-1}\\
\a_{2}(t)&=&t^{2}d_{0,0}^{2}-2td_{0,0}+1+2t^{2}d_{-1,0}d_{1,0}-4t^{2}d_{-1,1}d_{1,-1}-4t^{2}d_{0,1}d_{0,-1}-4t^{2}d_{1,1}d_{-1,-1}\\
\a_{3}(t)&=&-2td_{1,0}+2t^{2}d_{0,0}d_{1,0}-4t^{2}d_{1,1}d_{0,-1}-4t^{2}d_{0,1}d_{1,-1}\\
\a_{4}(t)&=&t^{2}d_{1,0}^{2}-4t^{2}d_{1,1}d_{1,-1}\\
\b_{0}(t)&=&t^{2}d_{0,-1}^{2}-4t^{2}d_{1,-1}d_{-1,-1}\\
\b_{1}(t)&=&2t^{2}d_{0,-1}d_{0,0}-2td_{0,-1}-4t^{2}d_{1,-1}d_{-1,0}-4t^{2}d_{1,0}d_{-1,-1}\\
\b_{2}(t)&=&t^{2}d_{0,0}^{2}-2td_{0,0}+1+2t^{2}d_{0,-1}d_{0,1}-4t^{2}d_{1,-1}d_{-1,1}-4t^{2}d_{1,0}d_{-1,0}-4t^{2}d_{1,1}d_{-1,-1}\\
\b_{3}(t)&=&-2td_{0,1}+2t^{2}d_{0,0}d_{0,1}-4t^{2}d_{1,1}d_{-1,0}-4t^{2}d_{1,0}d_{-1,1}\\
\b_{4}(t)&=&t^{2}d_{0,1}^{2}-4t^{2}d_{1,1}d_{-1,1}.
\end{array}$$ 

 Note that $\Delta_{1}([x_0:x_1])$ (resp. $\Delta_{2}([y_0:y_1])$)  is of degree 4 and so has four roots counted with multiplicities $a_1, a_2, a_3, a_4$ (resp. $b_1, b_2, b_3 , b_4$) in $\P1(\C)$. If the  discriminant $\Delta_{1}([x_0:x_1])$ (resp. $\Delta_{2}([y_0:y_1])$) has a double root; up to renumbering, we can assume that $a_1 =a_2$ (resp. $b_1 =b_2$).

 \begin{lem}\label{lem:disczeroes} 
 Assume that the {\it model of the walk} is {\it nondegenerate} and belongs to the first family of \eqref{the five step set}
$$\begin{tikzpicture}[scale=.6, baseline=(current bounding box.center)]
\foreach \x in {-1,0,1} \foreach \y in {-1,0,1} \fill(\x,\y) circle[radius=0pt];
\draw[thick,->](0,0)--(-1,1);
\draw[thick,->](0,0)--(1,-1);
\draw[dashed,->,line width=0.5mm](0,0)--(1,1);
\draw[dashed,->,line width=0.5mm](0,0)--(1,0);
\draw[dashed,->,line width=0.5mm](0,0)--(0,1);
\end{tikzpicture}$$
 Then, the {\it walk} has genus zero and the singular point of $\Etproj$ is $\Omega = ([0:1],[0:1])$, that is, $a_1 = a_2 = [0:1]$  (resp. $b_1 = b_2 = [0:1]$) is a double root of $\Delta_{1}([x_0:x_1])$ (resp. $\Delta_{2}([y_0:y_1])$). The other roots are distinct from one another and from the double root and are given by 
{\arraycolsep=1.4pt\def\arraystretch{2.0}
$$
\begin{array}{|l|l|l|}\hline
&a_1=a_2&b_1 =b_2\\\hline
&[0:1]&[0:1]\\\hline
&a_{3}&a_{4}\\[0.05in]\hline
\a_{4}(t)\neq 0&\left[-\a_{3}(t)-\sqrt{\a_{3}(t)^{2}-4\a_{2}(t)\a_{4}(t)}:2\a_{4}(t)\right]&\left[-\a_{3}(t)+\sqrt{\a_{3}(t)^{2}-4\a_{2}(t)\a_{4}(t)}:2\a_{4}(t)\right]
\\[0.05in]\hline
\a_{4}(t)= 0&[1:0] &[-\a_{2}(t):\a_{3}(t)]\\ \hline
&b_{3}&b_{4}\\[0.05in]\hline
\b_{4}(t)\neq 0&\left[-\b_{3}(t)-\sqrt{\b_{3}(t)^{2}-4\b_{2}(t)\b_{4}(t)}:2\b_{4}(t)\right]&\left[-\b_{3}(t)+\sqrt{\b_{3}(t)^{2}-4\b_{2}(t)\b_{4}(t)}:2\b_{4}(t)\right]
\\[0.05in]\hline
\b_{4}(t)= 0&[1:0] &[-\b_{2}(t):\b_{3}(t)]\\[0.05in] \hline
\end{array} $$}
\end{lem} 

\begin{rmk}\label{rem1}
We can extend Lemma \ref{lem:disczeroes} to the other configurations in \eqref{the five step set} by using the following remarks:
\begin{enumerate}
\item Replacing $([x_0 : x_1],[y_0 : y_1])$ by $([x_0 : x_1],[y_1 : y_0])$, which corresponds to the variable change $(x,y)\mapsto (x,y^{-1})$,  amounts to consider a {\it weighted model of walk} with  {\it weights} $d'_{i,j}:=d_{i,-j}$. This can be used to extend Lemma \ref{lem:disczeroes}  to the second configuration of  \eqref{the five step set}; for instance, the singular point of $\Etproj$ is $\Omega=([0 : 1],[1 :0])$ in that case. 
\item Replacing $([x_0 : x_1],[y_0 : y_1])$ by $([x_1 : x_0],[y_1 : y_0])$ amounts to consider a {\it weighted model of walk} with  {\it weights} $d'_{i,j}:=d_{-i,-j}$. This can be used to extend Lemma \ref{lem:disczeroes}  to the third configuration of \eqref{the five step set}; for instance, the singular point of $\Etproj$ is $\Omega=([1 : 0],[1 :0])$ in that case. 
\item 
Replacing $([x_0 : x_1],[y_0 : y_1])$ by $([x_1 : x_0],[y_0 : y_1])$ amounts to consider a {\it weighted model of walk} with  {\it weights} $d'_{i,j}:=d_{-i,j}$. This can be used to extend Lemma \ref{lem:disczeroes}  to the fourth configuration of \eqref{the five step set}; for instance, the singular point of $\Etproj$ is $\Omega=([1 : 0],[0 :1])$ in that case.
\end{enumerate}
\end{rmk}

\begin{rmk}
Note that if we consider the $x_3,x_4$ (resp. $y_3,y_4$) defined in \cite[Chapter~6]{FIM}, we have the equality of sets $\{a_3,a_4\}=\{x_3,x_4\}$ and $\{b_3,b_4\}=\{y_3,y_4\}$, but do not have necessarily $a_i=x_i$, $b_j=y_j$, with $3\leq i,j \leq 4$.
\end{rmk}

\begin{proof}[Proof of Lemma \ref{lem:disczeroes}] We shall prove  the lemma for $\Delta_{2}([y_0:y_1])$, the proof for $\Delta_{1}([x_0:x_1])$ being similar. By assumption, $d_{-1,-1}=d_{-1,0}=d_{0,-1}=0$. Then, $\alpha_{0}(t)=\alpha_{1}(t)=0$. Therefore, the discriminant $\Delta_{2}([y_0:y_1])$  has a double root at $[0:1]$ and we can write
$$\begin{array}{lll}\Delta_{2}([y:1])& = &\b_2(t) y^2 + \b_3(t)y^3 + \b_4(t) y^4.\end{array}$$
Since $t$ is transcendental over $\Q(d_{i,j})$, we see that the coefficient of $y^2$ is nonzero.  Therefore $ [0:1]$ is precisely a double root of $\Delta_{2}([y_0:y_1])$.   To see that $b_3$ and $b_4$ are distinct, we calculate the discriminant of $\Delta_{2}([y:1])/y^2$, which is almost the same as the one we considered in the proof of Lemma \ref{theo:doublezero}.
 This is a polynomial of degree four in $t$ with the following coefficients:
{\arraycolsep=1.4pt\def\arraystretch{1.1
}
$$\begin{array}{|c|c|}
\hline
\text{term}& \text{coefficient} \\\hline
t^4 &-16(4d_{-1,1}d_{1,-1}d_{1,1}-d_{1,-1}d_{1,0}^2 - d_{0,0}^2d_{1,1} +d_{0,0}d_{0,1}d_{1,0}-d_{0,1}^2d_{1,-1})d_{-1,1}\\ \hline
t ^3& -16(2d_{0,0}d_{1,1} - d_{0,1}d_{1,0})d_{-1,1}\\\hline
t^2& 16 d_{-1,1}d_{1,1}\\ \hline
t& 0\\ \hline
1& 0\\ \hline
\end{array}$$}
Let us prove that if $\Delta_{2}([y_0:y_1])$ has a double root different from $ [0:1]$, all the above coefficients must be zero. Recalling that $d_{-1,1}d_{1,-1} \neq 0$, from the coefficient of $t^2$, we must have $d_{1,1} =0$. From the coefficient of $t^3$, we have that $d_{0,1} =0$ or $d_{1,0} = 0$. From the coefficient of $t^4$, we get in both cases $d_{0,1}=d_{1,0} = 0$.  This implies that the  model of the {\it walk} would be {\it degenerate}, a contradiction.  
The formulas for $b_3$ and $b_4$ follow from the quadratic formula.\end{proof}
\section{Involutive automorphisms of the kernel curve}\label{appendix:involutions}

Following \cite[Section 3]{BMM}, \cite[Section 3]{KauersYatchak}  or \cite{FIM}, we consider the involutive rational functions 
$$ 
i_{1}, i_{2} : \C^{2} \dashrightarrow \C^{2}
$$ 
 given by
$$
i_1(x,y) =\left(x, \frac{A_{-1}(x) }{A_{1}(x)y}\right) \text{ and }  i_2(x,y)=\left(\frac{B_{-1}(y)}{B_{1}(y)x},y\right).
$$ 
Note that $i_{1}, i_{2}$ are ``only'' rational functions in the sense that they are  a priori not defined when the denominators vanish. The dashed arrow notation used above and in the rest of the paper is a classical notation for rational functions.

 The rational functions $i_{1},i_{2}$ induce involutive rational functions
$$ 
\iota_{1}, \iota_{2} : \Etproj \dashrightarrow \Etproj 
$$ 
given by 
$$
\begin{array}{llll}
&\iota_1([x_0: x_1],[y_0:y_1]) &=&\left([x_0: x_1], \left[\dfrac{A_{-1}(\frac{x_{0}}{x_{1}}) }{A_{1}(\frac{x_{0}}{x_{1}})\frac{y_{0}}{y_{1}}}:1\right]\right),\\ \text{ and } & \iota_2([x_0: x_1],[y_0:y_1])&=&\left(\left[\dfrac{B_{-1}(\frac{y_{0}}{y_{1}})}{B_{1}(\frac{y_{0}}{y_{1}})\frac{x_{0}}{x_{1}}}:1\right],[y_0:y_1]\right).
\end{array}
$$ 
Again, these functions are  a priori not defined where the denominators vanish. However, the following result shows that, actually, this is only an ``apparent problem'': $\iota_{1}$ and $\iota_{2}$ can be extended {into endomorphisms} {of $\Etproj$}. We recall that a rational map $f :\Etproj \dashrightarrow \Etproj$ is an endomorphism if it is regular at any $P \in \Etproj$,  i.e. if $f$ can be represented in suitable affine charts containing $P$ and  $f(P)$  by a rational function with nonvanishing denominator at $P$. More generally, given $X$ and $Y$ 	algebraic varieties, we say that $f :X \dashrightarrow Y$ is a morphism if  $f$ can be represented by two suitable affine charts containing $P$ and $f(P)$ respectively, by a rational function with nonvanishing denominator at $P$.

\begin{prop}
The rational maps 
$ 
 \iota_{1}, \iota_{2} : \Etproj \dashrightarrow \Etproj  $
can be extended into involutive automorphisms of $\Etproj$. 
\end{prop}

\begin{proof} 
 Note that $\iota_{1}(x,y)$ is well-defined if the $x_{i}$ and the $y_{j}$ are nonzero and if $A_{1}(\frac{x_{0}}{x_{1}})\frac{y_{0}}{y_{1}}$ is nonzero. This excludes at most finitely many $(x,y) \in \Etproj$ and, hence,  there exists a finite set $\mathcal{S}_{0}\subset\Etproj$ such that $\iota_{1}$ is well defined on $\Etproj\setminus \mathcal{S}_{0}$. The map $\iota_{1}$ induces a bijection from $\Etproj\setminus \mathcal{S}_{0}$ to $\Etproj\setminus \mathcal{S}_{1}$, where $\mathcal{S}_{1}$ is a finite set. The same holds for $\iota_{2}$. We have to prove that $ 
 \iota_{1}, \iota_{2} : \Etproj \dashrightarrow \Etproj  $
can be extended into endomorphisms of $\Etproj$.
According to Proposition~\ref{prop:genuszeroKernel}, if the curve $\Etproj$ has genus one, then it is smooth and the result follows from \cite[Proposition 6.8, p.~43]{Hart}.

It remains to study the case when $\Etproj$ has genus zero. In that case, Proposition \ref{prop:genuszeroKernel} ensures that $\Etproj$ has a unique singularity $\Omega$.  It follows from \cite[Proposition 6.8, p.~43]{Hart} that $\iota_{1}$ and $\iota_{2}$ can be uniquely extended into morphisms $\Etproj \setminus \{\Omega\} \rightarrow \Etproj$ still denoted by $\iota_{1}$ and $\iota_{2}$. It remains to study $\iota_{1}$ and $\iota_{2}$ at $\Omega$. Let us first assume that the {\it walk} under consideration belongs to the family of the first configuration of \eqref{the five step set}. Lemma \ref{lem:disczeroes} ensures that $\Omega=([0:1],[0:1])$.  For $([x:1],[y:1]) \in \Etproj$, the equation $K(x,y,t)=0$ implies that 
\begin{equation}\label{for iota1 morph}
\frac{A_{-1}(x)}{A_{1}(x)y} = \frac{1}{t A_{1}(x)} - \frac{A_{0}(x)}{A_{1}(x)} - y = \frac{x}{t \widetilde{A}_{1}(x)} - \frac{\widetilde{A}_{0}(x)}{\widetilde{A}_{1}(x)} - y
\end{equation}
where $\widetilde{A}_{0}(x)=xA_0(x)=d_{-1,0} + d_{0,0}x+d_{1,0}x^{2}$ and $\widetilde{A}_{1}(x)=xA_1(x)=d_{-1,1} + d_{0,1}x+d_{1,1}x^{2}$. Since $d_{-1,1} \neq 0$, $\widetilde{A}_{1}(x)$ does not vanish at $x=0$. Since $d_{-1,0}=0$,  $\widetilde{A}_{0}(x)$ vanishes at $x=0$.  So, \eqref{for iota1 morph} shows that $\iota_{1}$ is regular at $\Omega$ and that $\iota_{1}(\Omega)=\Omega$. The argument for $\iota_{2}$ is similar. \par 
The other cases listed in \eqref{the five step set} can be treated similarly using a reduction argument as in Remark \ref{rem1}.
\end{proof}

\begin{figure}
\begin{center}
    


\includegraphics[scale=1]{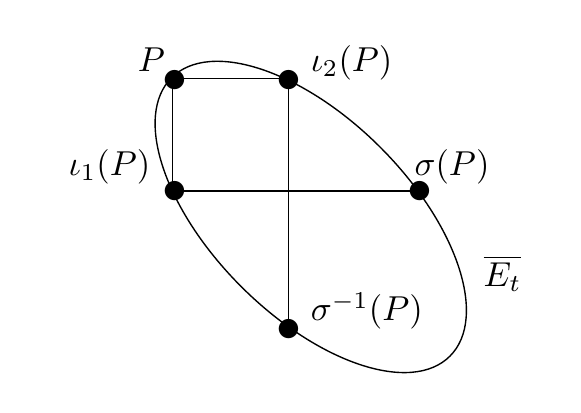}
\caption{The maps $\iota_{1},\iota_{2}$ restricted to the {\it kernel curve} $\Etproj $}\label{figiota}
\end{center}
\end{figure}

We also consider the automorphism of $\Etproj$ defined by 
$$
\sigma=\iota_2 \circ \iota_1. 
$$ 
It is easily seen that $\iota_{1}$ and $\iota_{2}$ are the vertical and horizontal  switches of $\Etproj$ (see Figure~\ref{figiota}),  i.e. for any $P=(x,y) \in \Etproj$, we have 
$$
\{P,\iota_1(P)\} = \Etproj \cap (\{x\} \times \P1(\C))
\text{ and }
\{P,\iota_2(P)\} = \Etproj \cap (\P1(\C) \times \{y\}).
$$
We now give a couple of lemmas for later use.

\begin{lemma}\label{lemma:fixedpointinvolution}
A point $P=([x_0: x_1],[y_0:y_1]) \in \Etproj$ is fixed by $\iota_1$ 
if and only if $\Delta_{1}([x_0:x_1])=0$. A point $P=([x_0: x_1],[y_0:y_1]) \in \Etproj$ is fixed by $\iota_2$ 
if and only if $\Delta_{2}([x_0:x_1])=0$.
\end{lemma}

\begin{proof}

Assume that $P$ is fixed by $\iota_1$.  Then, the polynomial $[y_0:y_1]\mapsto \overline{K}(x_0,x_1,y_0,y_1,t)$ has a double root, meaning that the discriminant is zero. This is exactly $\Delta_{1}([x_0:x_1])=0$.  Conversely,  $\Delta_{1}([x_0:x_1])=0$ implies that $[y_0:y_1]\mapsto \overline{K}(x_0,x_1,y_0,y_1,t)$ has a double root and therefore $P$ is fixed by $\iota_{1}$. The proof for $\iota_{2}$ is similar.  \end{proof}

The fixed points of $\iota_1$ have $y$-coordinates that are the double roots of $y \mapsto \overline{K}(x_0,x_1,y,t)$, i.e. they are the roots of the discriminant. By Lemma \ref{lem:disczeroes} and Remark \ref{rem1}, there are $3$ points of $\Etproj$ that are fixed by $\iota_{1}$.
A similar statement holds for $\iota_2$. As is shown in  the following lemma, one of them plays a particular role.

\begin{lem}\label{lem:dynamicgenuszero}
Let $P \in \Etproj$. The following statements are equivalent:
\begin{enumerate}
\item\label{lem:dynamicgenuszero1} $P$ is fixed by $\iota_1$ and $\iota_2$;
\item\label{lem:dynamicgenuszero2} $P$ is a singular point of $\Etproj$;
\item\label{lem:dynamicgenuszero3} $P$ is fixed by $\sigma=\iota_2 \circ \iota_1$.
\end{enumerate}
\end{lem}

\begin{proof}
Let $P=([a:b],[c:d])\in \Etproj$.  From Proposition~\ref{prop:genuszeroKernel}, we have that $P$ is a singular point if and only if 
$\Delta_{1}([x_0:x_1])$ and $\Delta_{2}([y_0:y_1])$ vanish at $[a:b]$ and  $[c:d]$ respectively. We conclude with Lemma~\ref{lemma:fixedpointinvolution}, that \eqref{lem:dynamicgenuszero1} is equivalent to \eqref{lem:dynamicgenuszero2}.\par 
Clearly, \eqref{lem:dynamicgenuszero1} implies \eqref{lem:dynamicgenuszero3}. It remains to prove that \eqref{lem:dynamicgenuszero3} implies \eqref{lem:dynamicgenuszero1}. Assume that ${P=(a_1,b_1)}$ is fixed by $\sigma$. After writing $\iota_1(P)=( a_1, b_1')$ and $\iota_2( \iota_1(P))=( a_1', b_1')$, it is clear that $\sigma(P)=P$ implies successively $\iota_1(P)=P$ and $\iota_2(P)=P$.
\end{proof}

\section{Uniformization of the kernel curve}\label{appendix:param}

We still consider a {\it weighted model of nondegenerate walk}. The aim of this section is to give an explicit uniformization of $\Etproj$. Thanks to  Proposition \ref{prop:genuszeroKernel}, the latter may have genus zero or one.   Although there are  algorithms to compute such uniformizations, see for instance \cite{van1997rational,sendra2008rational}, our presentation of explicit uniformizations allows us to understand in detail the pull-backs of $\sigma$, $\iota_{1}$ and $\iota_{2}$ and therefore their effect on the generating series of the models of walks. Let us start with the genus zero case. 

\subsection{Genus zero case}
Let us consider a {\it nondegenerate weighted model of walks} of genus zero. Thank to Corollary \ref{cor1} combined with Remark \ref{rem1}, it suffices to consider the situation where the {\it nondegenerate model of walk} arises from the following family
$$
\begin{tikzpicture}[scale=.6, baseline=(current bounding box.center)]
\foreach \x in {-1,0,1} \foreach \y in {-1,0,1} \fill(\x,\y) circle[radius=0pt];
\draw[thick,->](0,0)--(-1,1);
\draw[thick,->](0,0)--(1,-1);
\draw[dashed,->,line width=0.5mm](0,0)--(1,1);
\draw[dashed,->,line width=0.5mm](0,0)--(1,0);
\draw[dashed,->,line width=0.5mm](0,0)--(0,1);
\end{tikzpicture}$$

  Genus zero curves may be parametrized with maps $\phi: \P1 (\C) \rightarrow \Etproj$ which are bijective outside a finite set. 
The aim of this subsection, achieved with Proposition~\ref{prop:parameterizationcompauto}, is to find  such a parametrization explicitly. Although we could have just written down the formula for this parametrization and verified its properties, we have preferred to explain how the formula arises. This requires a preliminary study of the automorphisms of $\P1(\C)$ obtained by pulling back the maps $\sigma$, $\iota_{1}$ and $\iota_{2}$ by $\phi$, which is done with a series of lemmas preceding Proposition~\ref{prop:parameterizationcompauto}.

According to Lemma \ref{lem:disczeroes}, $\Etproj$ has a unique singular point $\Omega=(a_{1},b_{1})=([0:1],[0:1])$. Moreover   
$\Delta_{1}([x_0:x_1])$ has degree four with a double root at $a_{1}=[0:1]$ and the remaining two roots $a_{3},a_{4}$ are distinct. We let $S_3$ and $S_4$ be the points of $\Etproj$ with first coordinate $a_3$ and $a_4$ respectively. 
 Similarly, $\Delta_{2}([y_0:y_1])$ has degree four with a double root at $b_{1}=[0:1]$ and the remaining two roots $b_{3},b_{4}$ are distinct. We let $S'_{3}$ and $S'_{4}$ be the points of $\Etproj$ with second coordinates $b_{3}$ and $b_{4}$ respectively.

\par 

Since $\Etproj$ has genus zero, there is a rational parametrization of $\Etproj$  \cite[Page 198, Ex.1]{Fultonalgcurves},\  i.e. there exists a birational map 
$$
\begin{array}{llll}
\phi = (\check{x},\check{y}):& \P1 (\C)& \dashrightarrow& \Etproj\\
&s&\mapsto &(\check{x}(s),\check{y}(s))
.\end{array}
$$
To simplify the notation, we will abusively denote $(\check{x},\check{y})$ by $(x,y)$.
We will now follow the ideas contained in \cite{FIM} to produce an explicit uniformization of $\Etproj$ in Proposition~\ref{prop:parameterizationcompauto}.  If we set $t=1$, we recover the uniformization of \cite[Section 6.4.3]{FIM}. However, it is not clear if their proof can be simply modified to hold in our context, so we preferred to give  proofs here with a slightly different strategy.

\begin{lem}\label{lem4}
The map $\phi$ is surjective and induces a bijection between ${\P1 (\C) \setminus \phi^{-1}(\Omega)}$ and $\Etproj \setminus \{\Omega\}$. 
\end{lem}
\begin{proof}
As any nonconstant rational map from $\P1(\C)$  to a projective curve, $\phi$ is actually a surjective morphism of curves, see \cite[Corollary 1, Page 160]{Fultonalgcurves}. Since $\Omega$ is the unique singular point of $\Etproj$, the result follows.
\end{proof}
 The maps $x,y : \P1 (\C) \rightarrow \P1 (\C)$ are surjective morphisms of curves as well.  

We let $s_{3},s_{4} \in \P1(\C)$ (resp. $s'_{3},s'_{4} \in \P1(\C)$) be such that $S_{3}=\phi(s_{3})$ and $S_{4}=\phi(s_{4})$ (resp. $S'_{3}=\phi(s'_{3})$ and $S'_{4}=\phi(s'_{4})$).

We will need to know the cardinality of $x^{-1}(P)$ (resp. $y^{-1}(P)$) for $P \in \P1(\C)$. This quantity might depend on $P$ but it is a general fact about morphisms of curves that the cardinality of $x^{-1}(P)$ (resp. $y^{-1}(P)$) is constant for $P$ outside a finite subset of $\P1(\C)$. This common value is called the degree of $x$ (resp. $y$).
Inside this finite set, the cardinality can only fall, so is less than the degree.

\begin{lem}\label{lem: x y deg 2}
The morphisms $x,y : \P1 (\C) \rightarrow \P1 (\C)$ have degree two.
\end{lem} 

\begin{proof}
We will see that this is a consequence of the fact that $\Etproj$ is a biquadratic curve. 
Observe that by Lemma \ref{lem4}, $\phi$ induces a bijection between $\P1 (\C) \setminus \phi^{-1}(\Omega)$ and $\Etproj \setminus \{\Omega\}$. Any $(a,b)\in \Etproj$ with $a\neq a_1$ cannot be $\Omega$ and therefore has a unique preimage by $\phi$. Additionally, let $Z$ be the finite set of zeros of the discriminant $\Delta_{1}$. Then, for any $a\notin Z$, the cardinality $(\{a\} \times \P1(\C)) \cap \Etproj$ is two. Since $x^{-1}(a)=\phi^{-1} ((\{a\} \times \P1(\C)) \cap \Etproj)$ and $a_1\in Z$, it follows that if $a\notin Z$, the cardinality of $x^{-1}(a)$ is two.
 So, $x$ has degree two. The argument for $y$ is similar. 
\end{proof}

Since $\phi$ is a birational map,  the involutive automorphisms $\iota_1,\iota_2$ of $\Etproj$ induce involutive automorphisms $\iup_{1},\iup_{2}$ of $\P1 (\C)$  via $\phi$. Similarly, $\sigma$ induces an automorphism $\tilde{\sigma}$ of $\P1(\C)$. In other words, we have the commutative diagrams 
 $$
\xymatrix{
    \Etproj  \ar@{->}[r]^{\iota_k} & \Etproj  \\
    \mathbb{P}^{1}(\C) \ar@{->}[u]^\phi \ar@{->}[r]_{\iup_k} & \mathbb{P}^{1}(\C) \ar@{->}[u]_\phi 
  }
  \text{ and }
  \xymatrix{
    \Etproj  \ar@{->}[r]^{\sigma} & \Etproj  \\
    \mathbb{P}^{1}(\C) \ar@{->}[u]^\phi \ar@{->}[r]_{\tilde{\sigma}} & \mathbb{P}^{1}(\C) \ar@{->}[u]_\phi 
  }
$$
{Note that since by Lemma \ref{lem4} $\phi$ induces a bijection between $\P1 (\C) \setminus \phi^{-1}(\Omega)$ and $\Etproj \setminus \{\Omega\}$  and $\Omega$ is fixed by $\iota_{1},\iota_{2}$, see Lemma \ref{lem:dynamicgenuszero}, the group generated by $\iota_{1}$ and $\iota_{2}$ is isomorphic to the group generated by $\iup_{1}$ and $\iup_{2}$. Thus we recover the same group as in \cite{BMM} for instance.} Note that although the cardinal of the group may depends upon $t$, see Remark~\ref{rem2}, since the maps $\iota_{1},\iota_{2}$ are defined on $\Q(d_{i,j})(t)$, two distinct values of $t$  transcendental  over $\Q(d_{i,j})$ lead to isomorphic groups.  We summarize some  remarks in the following lemmas.
  
 \begin{lemma}\label{lem8bis}
We have  $x=x \circ \iup_1$ and $y=y \circ \iup_2$.
\end{lemma}

\begin{proof}
We obtain $x=x \circ \iup_1$ by equating the first coordinates in the equality $\phi \circ \iup_{1} = \iota_{1} \circ \phi$ and we obtain $y=y \circ \iup_2$ by equating the second coordinates in the equality ${\phi \circ \iup_{2} = \iota_{2} \circ \phi}$. 
\end{proof}

\begin{lemma}\label{lemma:nonsingularfixedpoints}
Let $P =\phi(s) \in \Etproj$ and let $k \in \{1,2\}$. We have:  
\begin{itemize}
\item if $\iup_k(s)= s$, then $\iota_k(P)=P$;
\item if $P \neq \Omega$ and $\iota_k(P)=P$, then $\iup_k(s)= s$.
\end{itemize}
Furthermore the map $\iup_1$ (resp. $\iup_2$) has exactly two fixed points, namely $s_{3}$ and $s_{4}$ (resp. $s'_{3}$ and $s'_{4}$).
\end{lemma}

\begin{proof}
We have $\iota_k(P)=\iota_{k} (\phi(s))=\phi (\iup_{k}(s))$. The first assertion is now clear, and the second one follows from the fact that $\phi$ is injective on $\Etproj \setminus \phi^{-1}(\Omega)$. Since $S_{3},S_{4}\neq \Omega$ are fixed by $\iota_{1}$, this shows that  $s_{3}$ and $s_{4}$ are fixed by $\iup_{1}$. Similar proof holds for $\iup_{2}$. \par 
It remains to prove that there are exactly two points fixed by $\iup_k$. To the contrary, assume that there is a third point fixed by $\iup_k$. Since $\iup_{k}$ is an automorphism of $\P1(\C)$,  i.e. an homography, with three fixed points, it is the identity. This is a contradiction and concludes the proof of the lemma.
\end{proof}

\begin{lemma}\label{lem8}

The preimage of $\Omega$ by $\phi$ has two elements. 
\end{lemma}

\begin{proof}

We know that $x,y : \P1 (\C) \rightarrow \P1 (\C)$ have degree two, so $\phi^{-1}(\Omega)$ has one or two elements. Suppose that $\phi^{-1}(\Omega)$ has exactly one element, say $s_{1}$. In virtue of $\phi(s_3)=S_3$ and $\phi(s_4)=S_4$, $s_1$  is different to $s_3,s_4$. Since $\phi (\iup_{1}(s_{1})) = \iota_{1} (\phi(s_{1}))=\iota_{1}(\Omega)=\Omega$, we have $\iup_{1}(s_{1}) =s_{1}$.  This contradicts Lemma~\ref{lemma:nonsingularfixedpoints}. Hence, $\phi^{-1}(\Omega)$ has two elements. 
\end{proof}

From now on, we define $s_{1} \neq s_{2}$ the two preimages  of $\Omega$ by $\phi$, that is $$\{s_{1},s_{2}\}:=\phi^{-1}(\Omega).$$

\begin{lem}\label{lem:fixedpointsdynamic}
The map $\iup_1$ (resp. $\iup_2$) interchanges $s_{1}$ and $s_{2}$.
The map $\tilde{\sigma}$ has exactly two distinct fixed points:  $s_{1}$ and $s_{2}$. 
\end{lem}

\begin{proof}
We have $\phi(s)=\Omega$ if and only if $s=s_{1}$ or $s_{2}$ and the equality $\iota_{1}(\phi(s))=\phi(\iup_{1}(s))$ shows that $\iup_{1}$ induces a permutation of $\phi^{-1}(\Omega)=\{s_{1},s_{2}\}$. By Lemma \ref{lemma:nonsingularfixedpoints}, $s_{1}$ is not fixed by  $\iup_{1}$, showing that the permutation is not the identity, i.e. $\iup_{1}$ interchanges $s_{1}$ and $s_{2}$.

The proof for $\iup_{2}$ is similar. 

As any homography which is not the identity, $\tilde{\sigma}$ has at most two fixed points in $\P1(\C)$. It only remains to prove that $s_1$ and $s_2$ are fixed by $\tilde{\sigma}$, and this is indeed the case because $\tilde{\sigma}=\iup_{2} \circ \iup_{1}$ and $\iup_{1},\iup_{2}$ interchange $s_{1}$ and $s_{2}$.
\end{proof}

\begin{figure}
\vspace{1cm}
\begin{tikzpicture}[scale=.8, baseline=(current bounding box.center)]

\draw (0,0) circle (3);

\draw (10,0) ..controls +(2,2) and +(1,0).. (10,3);
\draw (10,0) ..controls +(-2,2) and +(-1,0).. (10,3);
\draw (10,0) ..controls +(2,-2) and +(1,0).. (10,-3);
\draw (10,0) ..controls +(-2,-2) and +(-1,0).. (10,-3);

\fill(-3,0) circle[radius=2pt];
\put (-62,-3) {{$0$}};
\fill(3,0) circle[radius=2pt];
\put (55,-3) {{$\infty$}};
\fill(0,3) circle[radius=2pt];
\put (-3,55) {{$1$}};
\fill(0,-3) circle[radius=2pt];
\put (-3,-65) {{$-1$}};
\fill(3*1.414/2,3*1.414/2) circle[radius=2pt];
\put (38,38) {{$\lambda$}};
\fill(3*1.414/2,-3*1.414/2) circle[radius=2pt];
\put (38,-42) {{$-\lambda$}};

\fill(10,0) circle[radius=2pt];
\put (235,0) {{$\Omega$}};
\fill(10,3) circle[radius=2pt];
\put (215,55) {{$S_3$}};
\fill(10,-3) circle[radius=2pt];
\put (215,-62) {{$S_4$}};
\fill(8.95,1.5) circle[radius=2pt];
\put (210,30) {{$S'_3$}};
\fill(8.95,-1.5) circle[radius=2pt];
\put (210,-35) {{$S'_4$}};

\draw[dashed,-latex](3,0) ..controls +(1,1) and +(-1,0.9).. (10,0);
\draw[dashed,-latex](-3,0) ..controls +(1,1) and +(-1,1.2).. (10,0);
\draw[dashed,-latex](0,3) ..controls +(1,1) and +(-1,1).. (10,3);
\draw[dashed,-latex](3*1.414/2,3*1.414/2) ..controls +(1,1) and +(-1,1).. (8.95,1.5);
\draw[dashed,-latex](3*1.414/2,-3*1.414/2) ..controls +(1,-1) and +(-1,-1).. (8.95,-1.5);
\draw[dashed,-latex](0,-3) ..controls +(1,-1) and +(-1,-1).. (10,-3);

\put (50,100) {{\Large{$\phi: \mathbb{P}^{1}(\C)\longrightarrow \Etproj$}}};
\end{tikzpicture}
\caption{An idealized representation of the uniformization map used in Proposition \ref{prop:parameterizationcompauto}. The left hand side represents the complex Riemann sphere and the right hand side the curve $\Etproj$, seen as an abstract complex algebraic curve. }\label{figbis}
\end{figure}
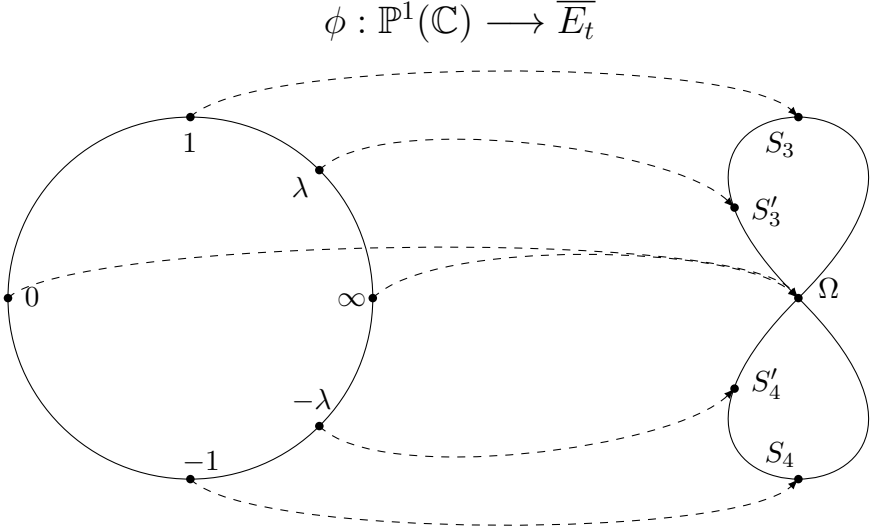

We are now ready to give an explicit expression of $\phi$. The coefficients $\a_{i},\b_{i}$ of the discriminants in this situation are given by the formulas 
$$\begin{array}{lllll}
\a_{0}(t)&=&\a_{1}(t)&=&0\\
\b_{0}(t)&=&\b_{1}(t)&=&0\\
\a_{2}(t)&=&\b_{2}(t)&=&1-2td_{0,0}+t^{2}d_{0,0}^{2}-4t^{2}d_{-1,1}d_{1,-1}\\
&&\a_{3}(t)&=&2t^{2}d_{1,0}d_{0,0}-2td_{1,0}-4t^{2}d_{0,1}d_{1,-1}\\
&& \b_{3}(t)&=&2t^{2}d_{0,1}d_{0,0}-2td_{0,1}-4t^{2}d_{1,0}d_{-1,1}\\
&&\a_{4}(t)&=&t^{2}(d_{1,0}^{2}-4d_{1,1}d_{1,-1})\\
&&\beta_{4}(t)&=&t^{2}(d_{0,1}^{2}-4d_{1,1}d_{-1,1}).
\end{array}$$ 
Note that for $k=3,4$, $\b_{k}(t)$, may be obtained from $\a_{k}(t)$ by interchanging $d_{1,0}$ with $d_{0,1}$ and $d_{1,-1}$ with  $d_{-1,1}$. 
\begin{prop}\label{prop:parameterizationcompauto}
An explicit parametrization $\phi : \P1 (\C) \rightarrow \Etproj$ such that 
$$
\iup_1(s)=\frac{1}{s},  \ {\iup_2 (s)= \frac{\lambda^{2}}{s}=\frac{q}{s} \text{ and } \widetilde{\sigma}(s) = qs
}$$
for a certain $\lambda \in \C^{*}$ is given by
$$\phi(s) 
=\left(\dfrac{4\a_{2}(t)}{\sqrt{\a_{3}(t)^{2}-4\a_{2}(t)\a_{4}(t)}( s +\frac{1}{s}) -2\a_{3}(t)}, 
\dfrac{4\b_{2}(t)}{\sqrt{\b_{3}(t)^{2}-4\b_{2}(t)\b_{4}(t)}( \frac{s}{\lambda}+\frac{\lambda}{s}) -2\b_{3}(t)}\right).$$
Moreover, we have, see Figure \ref{figbis}
$$\begin{array}{lll}
x(0)=x(\infty)=a_{1},&x(1)=a_3,&x(-1)=a_4,\\
 y(0)=y(\infty)=b_1, &y(\lambda)=b_3, &y(-\lambda)=b_4.
\end{array} $$
\end{prop}

{\begin{proof}
According to Lemma \ref{lem:fixedpointsdynamic}, $\iup_1$ is an involutive homography with fixed points $s_{3}$ and $s_{4}$, so there exists an homography $h$ such that $h(s_{3})=1$, $h(s_{4})=-1$ and ${h \circ \iup_{1} \circ h^{-1}(s)=1/s}$. Up to replacing $\phi$ by $\phi \circ h$, we can assume that $s_{3}=1$, $s_{4}=-1$ and $\iup_1(s)=\frac{1}{s}$. Since $s_{1} \neq s_{2}$, we can assume up to renumbering that $s_{1} \neq \infty$. Let us consider the homography
$k(s)=\frac{s-s_{1}}{-s_{1} s+1}$. Note that $k$ commutes with $s\mapsto 1/s$, so changing $\phi$ by $\phi \circ k$ does not affect $\iup_{1}$, and we can also assume that $s_{1}=[0:1]$ and $s_{2}=[1:0]$.  Lemma~\ref{lem: x y deg 2} and Lemma \ref{lem8bis} ensure that the morphism $x : \P1 (\C) \rightarrow \P1(\C)$ has degree two and satisfies $x(s) =x(1/s)$ for all $s \in \P1(\C)$.  
Since the morphism $x :  \P1(\C)\rightarrow  \P1(\C)$ has degree two, see Lemma \ref{lem: x y deg 2}, it follows that 
$$
x(s)=\frac{a(s +1/s) +b}{c (s+1/s)  + d  }
$$ 
for some $a,b,c,d \in \C$. 
We have $x(s_{1})=x([0:1])=a_{1}=0$, $x(s_{2})=x([1:0])=a_{1}=0$,  ${x(s_{3})=x([1:1])=a_{3}}$ and $x(s_{4})=x([-1:1])=a_{4}$. The equality $x([1:0])=0$ implies $a=0$. The equalities $x([1:1])=a_{3}$ and $x([-1:1])=a_{4}$ imply 
$$
x(s)= \frac{ 4a_{3}a_{4}}{(a_{4}-a_{3})( s +\frac{1}{s}) +2(a_{3}+a_{4})}. 
$$
The known expressions for $a_{3}$ and $a_{4}$ given in Lemma~\ref{lem:disczeroes} lead to the expected expression for $x(s)$. \par

According to Lemma \ref{lem:fixedpointsdynamic}, $\iup_2$ is an homography interchanging $[0:1]$ and $[1:0]$, so $\iup_2 (s)= \frac{\lambda^{2}}{s}$ for some $\lambda \in \C^{*}$.
Up to renumbering, we have $s'_{3}=\lambda$ and $s'_{4}=-\lambda$.
Using the fact that the morphism $y : \P1(\C) \rightarrow \P1(\C)$ has degree two and is invariant by $\iup_2$, and arguing as we did above for $x$, we see that there exist $\alpha,\beta,\gamma,\eta \in \C$ such that 
$$
y(s)=\frac{\alpha (\frac{s}{\lambda}+\frac{\lambda}{s}) +\beta }{\gamma(\frac{s}{\lambda}+\frac{\lambda}{s})+ \eta  }.
$$
The equality $y([1:0])=0$ implies $\a=0$. Using the equalities $y(s'_{3})=y(\lambda)=b_{3}$ and ${y(s'_{4})=y(-\lambda)=b_{4}}$, and arguing as we did above for $x$, we obtain the expected  expression for $y(s)$.  
\end{proof}

\begin{rem}\begin{enumerate}
 \item The uniformization is not unique.  More precisely, the possible uniformizations are of the form $\phi\circ h$, where $h$ is an homography. However, if one requires that $h$ fixes setwise $\{[0:1],[1:0]\}$ then $q$ is uniquely defined up to inversion.
\item The real $q$ or $q^{-1}$ specializes for $t=1$ to the real $\rho^2$ in \cite[Page 178]{FIM}.
\end{enumerate}
\end{rem}

 The following proposition determines $q$ up to its inverse. We include this for completeness.

\begin{prop}[\cite{dreyfus2017walks}, Proposition 1.7, Corollary 1.10]\label{lem:qBIS}
One of the two complex numbers $q$ or $q^{-1}$ is equal to 
$$ 
\dfrac{-1+d_{0,0}t-\sqrt{(1-d_{0,0}t)^{2}-4d_{1,-1}d_{-1,1}t^{2}}}{-1+d_{0,0}t+\sqrt{(1-d_{0,0}t)^{2}-4d_{1,-1}d_{-1,1}t^{2}}}.
$$
Furthermore, $q\in \R \setminus \{ \pm 1 \}$. 
\end{prop}

\begin{rem}
This implies that $\sigma$ and $\tilde{\sigma}$ have infinite order (see also \cite{BMM,fayolleRaschel}). 
Because  $\sigma$ is induced on $\Etproj$ by $i_{1} \circ i_{2}$, we find that $i_{1} \circ i_{2}$ has infinite order as well.
It follows that the {\it group of the walk} introduced in \cite{BMM}, which is by definition the group generated by $i_{1}$ and $i_{2}$, has infinite order.  Note that this was proved in \cite{BMM} using a valuation argument. 
\end{rem}

\begin{rem}
We stress  the fact that since $\phi(s)$, $q$ and $\Etproj$ depend continuously on $t$ and the set of transcendental number over $\Q(d_{i,j})$ in $]0,1[$ is dense in $]0,1[$, we deduce that Proposition \ref{prop:parameterizationcompauto}  and Proposition \ref{lem:qBIS} stay valid for every $t\in ]0,1[$.
\end{rem}

\subsection{Genus one case} 
In this section, we consider the uniformization problem in the genus one context. This problem has been solved in 
 \cite{dreyfus2017differential}. We recall below the main result of \cite{dreyfus2017differential}, for the sake of completeness.  
Let us consider a {\it nondegenerate model of walk} of genus one. 
 By Proposition \ref{prop:genuszeroKernel},
 $\Etproj$ is a smooth curve of genus one and, by Corollary \ref{cor1}, this corresponds to the situation where the step  set is not included in any half plane whose boundary passes through $(0, 0)$.
By \cite[Chapter XX]{WW}, it is biholomorphic to $\C/(\Z\omega_1 + \Z\omega_2)$ for some lattice $\Z\omega_1 + \Z\omega_2$ of $\CX$  via some $(\Z\omega_1 + \Z\omega_2)$-periodic holomorphic map 
\begin{equation}
\label{eq:Lambda}
\begin{array}{llll}
\Lambda :& \C& \rightarrow &\overline{E}_t\\
 &\omega &\mapsto& (\mathfrak{q}_1(\omega), \mathfrak{q}_2(\omega)),
\end{array}
\end{equation}
where $\mathfrak{q}_1, \mathfrak{q}_2$ are rational functions of $\wp$ and its derivative $d\wp/d\omega$, and $\wp$ is the Weierstrass function associated with the lattice $\Z\omega_1 + \Z\omega_2$:
\begin{equation}
\label{eq:expression_wp_expanded}
     \wp(\omega )=\wp(\omega;\omega_1,\omega_2):=\frac{1}{\omega^{2}}+ \sum_{(\ell_{1},\ell_{2}) \in \Z^{2}\setminus \{(0,0)\}} \left(\frac{1}{(\omega +\ell_{1}\omega_{1}+\ell_{2}\omega_{2})^{2}} -\frac{1}{(\ell_{1}\omega_{1}+\ell_{2}\omega_{2})^{2}}\right).
\end{equation}     
Then, the field of meromorphic functions on $\Etproj$ is isomorphic to the field of meromorphic functions on $\C/(\Z\omega_1 + \Z\omega_2)$, which is itself isomorphic to the field of meromorphic functions on $\C$ that are $(\omega_{1},\omega_{2})$-periodic. The latter field is equal to $\C(\wp, \wp')$, see \cite{WW}.
 
The maps $\iota_{1}$, $\iota_{2}$ and $\sigma$ may be lifted to the $\omega$-plane $\CX$. We denote these lifts by  $\iup_{1}$, $\iup_{2}$ and $\widetilde{\sigma}$ respectively. So we have the commutative diagrams 
\begin{equation*}
\xymatrix{
    \Etproj  \ar@{->}[r]^{\iota_k} & \Etproj  \\
    \C \ar@{->}[u]^\Lambda \ar@{->}[r]_{\iup_k} & \C \ar@{->}[u]_\Lambda 
  }
  \qquad\qquad\qquad
  \xymatrix{
    \Etproj  \ar@{->}[r]^{\sigma} & \Etproj  \\
\C \ar@{->}[u]^\Lambda \ar@{->}[r]_{\widetilde{\sigma}} & \C \ar@{->}[u]_\Lambda 
  }  
\end{equation*}
 
The following result has been proved
\begin{itemize}
\item in \cite[Section 3.3]{FIM} when $t=1$,
\item in \cite{RaschelJEMS} in the {\it unweighted} case for general $0<t<1$, not necessarily transcendental over $\Q(d_{i,j})$,
\item in  \cite[Proposition 18]{dreyfus2017differential} in the {\it weighted} case, with general $0<t<1$, not necessarily transcendental over $\Q(d_{i,j})$.
\end{itemize}
 In what follows, we set $D(\omega)=\Delta_{1}([\omega:1])$. Let us  introduce $z=2A(x)y+B(x)$, where $A(x)=t(d_{-1,1} + d_{0,1} x + d_{1,1}x^2)$, and 
$B(x)=t  (d_{-1,0} -  \frac{1}{t} x +d_{0,0}x + d_{1,0}x^2)$.

\begin{prop}
\label{prop:uniformization}
An explicit uniformization $\Lambda : \C \rightarrow \Etproj$ such that 
$$\iup_{1}(\omega)=-\omega, \quad\iup_{2}(\omega)=-\omega+\omega_{3}\quad \text{and} \quad\widetilde{\sigma}(\omega)=\omega+\omega_{3},$$
for a certain $\omega_{3}\in \C^*$ is given by 
$$
\Lambda(\omega)=(x(\omega),y(\omega))
$$
where $x(\omega)$ and $y(\omega)$ are given by 
\begin{equation*}
    \begin{array}{|l|c|c|}\hline 
&x(\omega)&z(\omega)\\\hline
a_{4}\neq [1\!:\!0]&\phantom{\Bigg\|}\left[a_{4}+\frac{D'(a_{4})}{\wp (\omega)-\frac{1}{6}D''(a_{4})}:1\right]\phantom{\Bigg\|}&\phantom{\Bigg\|}\left[\frac{D'(a_{4})\wp'(\omega)}{2(\wp(\omega)-\frac{1}{6}D''(a_{4}))^{2}}:1\right]\phantom{\Bigg\|}\\\hline 
a_{4}= [1\!:\!0]&\phantom{\Big\|}\left[\wp(\omega)-\alpha_{2}/3:\alpha_{3}\right]\phantom{\Big\|}&\phantom{\Big\|}\left[-\wp'(\omega):2\alpha_{3}\right]\phantom{\Big\|}\\\hline 
\end{array}
\end{equation*}
A suitable choice for the periods $(\omega_1,\omega_2)$ is given by the elliptic integrals
\begin{equation}
\label{eq:expression_omega_1_omega_2}
     \omega_{1}=\mathbf{i}\int_{a_{3}}^{a_{4}} \frac{dx}{\sqrt{\vert D(x)\vert}}\in \mathbf{i}\R_{>0}\quad\text{and}\quad
\omega_{2}=\int_{a_{4}}^{a_{1}} \frac{dx}{\sqrt{D(x)}}\in \R_{>0}.
\end{equation} 
\end{prop}

Note that, according to \cite[Section 2]{dreyfus2017differential},
$$\omega_{3}=\int_{a_{4}}^{X_{\pm}(b_{4})} \frac{dx}{\sqrt{D(x)}}\in ]0,\omega_{2}[.$$

\begin{rmk}\label{rem2}
Contrary to the genus zero situation, the map $\sigma$ may have finite order.
\end{rmk}

\bibliography{walkbib}
\end{document}